\documentclass[a4paper,10pt]{article}

\usepackage{verbatim} %
\usepackage{amssymb,amsmath,graphics}
\usepackage{amsthm}
\usepackage{mathtools}
\usepackage{todonotes}
\usepackage{bbm}
\newtheorem{theo}{Theorem}[section]
\newtheorem{prop}[theo]{Proposition}  
\newtheorem{coro}[theo]{Corollary}

\newtheorem{rema}[theo]{Remark}
\newtheorem{lema}[theo]{Lemma}
\newtheorem{defi}[theo]{Definition}
\newtheorem{fak}[theo]{Fact}

\newcommand{\klas}{\mathcal{W}(j,t,(Z(k))_{k=1}^n)}
\newcommand{\1}{\mathbf{1}}

\newcommand{\cons}{K}

\def\be#1\ee{\begin{equation}#1\end{equation}}
\newcommand{\ba}{\begin{eqnarray} }
\newcommand{\ea}{\end{eqnarray} }
\newcommand{\lv}{\left\lVert}
\newcommand{\rv}{\right\rVert}

\def\bt#1\et{\begin{theo}#1\end{theo}}
\def\bl#1\el{\begin{lema}#1\end{lema}}
\def\bp#1\ep{\begin{prop}#1\end{prop}}
\def\bd#1\ed{\begin{defi}#1\end{defi}}

 \global\long\def\cbr#1{\left\{  #1\right\}  }
 \global\long\def\rbr#1{\left(#1\right)}
 
\def\ccE{{\cal E}}
\def\ccF{{\cal F}}
\def\ccG{{\cal G}}

\def\va{\varepsilon}
\def\ra{\rightarrow}

\def\E{\mathbf{E}}

\def\P{\mathbf{P}}

\def\N{{\mathbb N}}

\def\R{{\mathbb R}}

\def\ls{\leqslant}
\def\gs{\geqslant}

\providecommand{\keywords}[1]{\textbf{\textit{Keywords: }} #1}

\providecommand{\klaso}[1]{\textbf{\textit{AMS MSC 2010: }} #1}

\begin{document}

\title{\bf  Small cover approach to the suprema of positive canonical processes\footnote{The authors were supported by National Science Centre, Poland grants 2019/35/B/ST1/04292 (W. Bednorz) and 2021/40/C/ST1/00330 (R. Meller). }.}
\author{Witold Bednorz\footnote{Institute of Mathematics, University of Warsaw, Banacha 2, 02-097 Warsaw, Poland.},  Rafał Martynek\footnote{Institute of Mathematics, University of Warsaw, Banacha 2, 02-097 Warsaw, Poland.} and Rafał Meller\footnote{Institute of Mathematics, University of Warsaw, Banacha 2, 02-097 Warsaw, Poland.; Institute of Mathematics, Polish Academy of Sciences, Śniadeckich 8,
00-656 Warsaw, Poland.}}
\maketitle

\begin{abstract}
We extend the recent result of Park and Pham on positive selector processes
\cite{Pa-Ph} to canonical processes generated by non-i.i.d. non-negative random variables satisfying minimal tail assumptions. We use it to derive a new comparison result for canonical stochastic processes based on variables that are uniformly distributed on generalized Orlicz balls (in particular on $\ell_p$ balls). We also conjecture that our result can be used to obtain new Sudakov-type bounds and show a result in this direction.\\
\keywords{expected suprema, positive processes, $p$-smallness.} \\
\klaso{60G15, 60G17.}
\end{abstract}

\section{Introduction}
One of the most influential results in the modern probability theory is Talagrand's Majorizing Measure Theorem concerning the expected value of the supremum of the Gaussian process \cite{TGGaus}. It describes exactly the value of $\E \sup_{t\in T} G_t$, where $(G_t)_{t\in T}$ is a centered, Gaussian stochastic process, in terms of the geometric properties of a metric space $(T,d)$ with $d^2(s,t)=\E (G_t-G_s)^2$. Not long after establishing this result, Talagrand discovered that the relation between the size of the process and the geometry of the index set can be extended to a broader class of processes. He was able to characterize the expectation of suprema of canonical processes based on variables with p.d.f given by $C(\alpha)\exp(-|x|^\alpha)$, $\alpha\gs 1$ \cite{TG2}. In addition, \cite{TalID} gives partial results on infinitely divisible processes were obtained together with some fundamental properties of the Bernoulli processes. These works from the early 90s can be considered as an informal beginning of a program aiming at characterizing the expected value of the suprema of canonical stochastic processes with deterministic quantities. R. Latała generalized Talagrand's result \cite{TG2} to variables with log-concave tails that do not decrease too fast \cite{Lat1}. Recently, it was discovered that it is sufficient to assume the appropriate moment asymptotic of the variables \cite{Lat2}, which makes it one of the most general results in the field.\\
\noindent At first, the problems mentioned above were tackled by means of the so-called majorizing measures, an idea due to X. Fernique. However, it soon turned out that a more natural tool was the chaining method, which is still used successfully today. Chaining, invented by Kolmogorov, was already a classical method when Talagrand published his seminal work \cite{TGGaus}. What was missing at the time was a change in the way of thinking about the chaining, which we owe to Talagrand.\\
Unfortunately, it seems that the chaining method is slowly reaching its limits. One of the first symptoms was the proof of Bernoulli's hypothesis \cite{Bernul} (characterization of $\E \sup_{t\in T} \sum_i t_i \va_i$, where $\va_1,\ldots$ are independent random signs). This result can be treated as the limiting case $\alpha=\infty$ of the canonical processes based on variables with the tails of order $\exp(-|x|^{\alpha})$. The only method used there was chaining, but the proof is very long and not easy to follow. The other problem with chaining is that it cannot take advantage of certain structural properties of the process, such as non-negativity. In his recent monograph \cite{TG4} Talagrand showed that the size of the expected suprema of empirical processes can be explained by decomposing the process into a sum of two processes. The expectation of the suprema of the first process can be described by the chaining method while the second process is positive. The analogous result was obtained for infinitely divisible processes \cite{ID}. This shows the need to create new tools for estimating the value of the expectation of suprema of non-negative canonical processes. \\
An alternative method for estimating the supremum of a process is the convex hull method (which is suitable for estimating any process, not necessarily a non-negative one). Although it is somehow related to the topic of this paper, we skip the details of this approach. The interested reader is referred to \cite{Lat3,TG4}. It is not entirely clear whether this method will be successful. If the problem can be solved by the chaining method, it is quite easy to construct the optimal convex hull. Unfortunately, there are currently no other tools in the literature to construct one.\\
In this work, we focus on the small cover approach proposed by Talagrand. His idea comes from the study of Gaussian processes. Let us denote such a process by $(G_t)_{t\in T}$, $T\subset \R^d$. It is easy to derive  from the main result of \cite{TGGaus} that there exists a numerical constant $L>0$ (a constant that does not depend on the process) and $t_1,t_2\ldots \in \R^d$ such that
\begin{align*}
\left\{\sup_{t\in T}|G_t| \gs L\E \sup_{t\in T} |G_t|   \right\} &\subset \bigcup_{k=1}^\infty \{|G_{t_k}| \gs 1 \},\\
\sum_{k\gs 1} \P(|G_{t_k}|  \gs 1) &\ls 1/2.
\end{align*}
In other words, there exists a simple witness of the fact that the process $(G_t)_{t\in T}$ cannot be too large.
Talagrand formulated a counterpart of this theorem for the selector process. He conjectured that there exists family $\ccG$ of subsets of $[d]$ such that
\begin{equation}\label{eq:selektory1}
\begin{split}
     \left\{ \sup_{t\in T} \sum_i t_i \delta_i \geq L \E \sup_{t\in T} \sum_i t_i \delta_i \right\} &\subset \bigcup_{W \in \ccG} \{ \forall_{i\in W} \delta_i=1\},\\
    \sum_{W\in \ccG} \P(\forall_{i\in W} \delta_i=1)&= \sum_{W\in \ccG} p^{|W|}\leq 1/2,  
\end{split}
\end{equation}
where $T\subset  \R^d_+$, and $(\delta_i)_{i\ls d}$ are i.i.d. Bernoulli r.v.'s with $p=\P(\delta_1=1)$.  We will refer to $\ccG$ as the small cover of the event $\left\{\sup_{t\in T}\sum_{i=1}^{d}t_i\delta_i\gs L\delta(T)\right\}$. Talagrand proved \eqref{eq:selektory1} for  $T\subset\{0,1\}^d$. His conjecture was recently confirmed by  Park and Pham \cite{Pa-Ph}. In the recent paper \cite{my} we provide an alternative proof of this fact and the counterpart for non-negative infinitely divisible processes. \\
This paper has two goals. The first is to extend the Park and Pham result to a broader class of processes. One of our main results,
Theorem \ref{thm:konc}, shows that if only $X_1,\ldots,X_d$ are nonnegative i.i.d. r.v.'s satisfying certain, undemanding, tail assumption, then there exist "simple" events $A_1, A_2\ldots \subset \R^d_+$ such that
\[\left\{\sup_{t\in T} \sum_{i=1}^d t_i X_i \gs L \E \sup_{t\in T} \sum_{i=1}^d t_i X_i \right\}\subset \bigcup_{i\geq i} A_i,\ \sum_i \P(A_i) \leq 1/2.\]
 Park and Pham's paper \cite{Pa-Ph} and our work \cite{my} were very combinatorial and based on the analysis of a set $\{i:\delta_i=1\}$. Interestingly, our argument from \cite{my}, with slight modifications, works for almost any canonical process. It can even be adapted to the non-i.i.d. case, which is shown by Theorem \ref{thm:wykl} below. Moreover, we show that small covers have practical applications. We have succeeded in using them to obtain a new comparison theorem for canonical processes (see Theorem \ref{thm:porow} below). This seems to be an evidence that small covers are not only a purely theoretical concept, but might serve as a powerful new tool in studying the expected suprema. We also conjecture that they can be used to derive new Sudakov-type bounds and show some results in this direction.

\noindent \textbf{Organization of the paper.} In Section \ref{sec:main} we state  the main results of this paper. In Theorem \ref{thm:wykl} we construct a small cover for non i.i.d. empirical processes (we also give an i.i.d. version of the theorem). Then we present the comparison theorem, which is a consequence of Theorem \ref{thm:wykl}. We show that the expectation of the supremum of the canonical process based on variables that are uniformly distributed on the generalized Orlicz balls (in particular on the $\ell_p$ balls) is controlled from the above by its independent copy. This can be seen as the most interesting result of this work. We also show that under additional assumptions on the r.v.'s there exist more natural small covers. In Section \ref{sec:app} we show that Sudakov-type bounds (in a 'toy' case) can be derived from the fact that small covers exist. We also prove the comparison result. In Section \ref{sec:pmain} we give a discrete version of Theorem \ref{thm:wykl} and show how it implies the non-discrete version and the Park and Pham result. Section \ref{sec:pdisc} is devoted to showing that small covers exist in the discrete case.  For the convenience of the reader, we collect definitions of all sets used in the proofs in the glossary, which is located before the bibliography.\\
\noindent \textbf{Notation.} Throughout this paper $\R_0=\R_+\cup \{0\}$. For the sake of notational consistency, when $x$ is a vector in $\R^d$, we denote its $i$-th coordinate by $x(i)$, instead of $x_i$.\\

\section{Main results}\label{sec:main}
The theorem below is the main result of this paper. For similar results concerning infinitely divisible processes, see \cite[Theorem 2.3]{my}.
\begin{theo}\label{thm:wykl}
 Fix $\delta\in (0,1)$, $K\geq 1$ and $T\subset \R_0^d$.
   Let $X_1,\ldots,X_d$ be i.i.d. non-negative random variables such that for some $C>1$ we have,
   \begin{equation}\label{eq:koncentracjaXniedyskretny}
\forall_{t \gs K\E X_i}\ \P(X_i\gs t)\gs C \P(X_i\gs 2t).
   \end{equation}
 Assume that $f_1,\ldots,f_d$ are strictly increasing, nonnegative functions defined on $\R_0$ such that for each $i\leq d$
\begin{align}
    \forall_{x\gs K\E X_1} f_i(2x) \leq C' f_i (x), \label{eq:efrosnie}\\
    \E f_i(X_1)\geq (C')^{-1} f_i(\E X_1), \label{eq:oczekdol}
\end{align}
for some $C'>1$.  Then there exists a constant $L=L(\delta,C,C',K)$ and a family $\ccG\subset \{(x_\ast,W):\ W\subset [d],\; x_\ast\in \R_0^{d}\}$ such that 
    \begin{align}
   & \forall_{(x_\ast,W) \in \ccG}\forall_{i\in W} x_\ast(i) \geq f_i(K\E X_i), \label{eq:swiadprzekex}\\
       & \ccF:=\left\{x\in \R^d_0:\sup_{t\in T} \sum_{i=1}^d t_i f_i(x(i)) \gs L \E \sup_{t\in T} \sum_{i=1}^d t_i f_i(X_i) \right\} \subset \bigcup_{(x_\ast,W)\in \ccG} \{x\in \R^d_0:\; \forall_{i\in W} f_i(x(i))\gs x_\ast(i)\},\label{eq:defef}\\
      &  \sum_{(x_\ast,W)\in \ccG} \P(\forall_{i\in W} f_i(X_i)\gs x_\ast(i) )\ls \delta \label{eq:locvenusaur2b}.
    \end{align}

\end{theo}
\noindent The above theorem fails for  $f_i \equiv 1$, $i=1,\ldots,d$. Thus, the assumption that $f_i$ are strictly increasing is essential.
\begin{rema}\label{rem:mniejwar}
Theorem \ref{thm:wykl} can be proved with \eqref{eq:koncentracjaXniedyskretny}  replaced by  
\[\forall_{t\gs K\E X_i }\ \P(X_i\gs t )\gs C \P(X_i\gs M t)\]
for some constant $M$. For the argument see Remark \ref{rem:zmiendow}.
\end{rema}
\noindent The condition \eqref{eq:koncentracjaXniedyskretny} ensures that the the function $t \mapsto \P(X_i \geq t)$ decays polynomially fast. This is far from being a concentration condition, since we can have $\P(X\gs t)=t^{-\alpha}$ for some $\alpha >1$ (to ensure that the expectation is well defined). Thus, the condition \eqref{eq:koncentracjaXniedyskretny} is undemanding. On the other hand, it fails for Bernoulli random variables. This case is covered by the discrete version of Theorem \ref{thm:wykl}, which we introduce in Section \ref{sec:pmain}.\\
A version of Theorem \ref{thm:wykl} for $i.i.d$ r.v.'s can be easily derived since $f(x)=x$ satisfies \eqref{eq:efrosnie},\eqref{eq:oczekdol} with $C'=2$ (where \eqref{eq:oczekdol} even holds for $C'=1$).
\begin{theo}\label{thm:konc}
Fix $\delta\in (0,1)$ and  $T\subset \R_0^d$.
   Let $X_1,\ldots,X_d$ be i.i.d. non-negative random variables satisfying \eqref{eq:koncentracjaXniedyskretny}. 
Then there exists a constant $L=L(\delta,C)$ and a family $\ccG\subset \{(x_\ast,W): \ x_\ast \in \R^d_0,\ W\subset [d]\}$ such that 
    \begin{align}
      & \forall_{(x_\ast,W) \in \ccG}\forall_{i\in W} x_\ast(i) \geq K\E X_i, \nonumber\\ 
        &\ccF:=\left\{x\in \R^d_0:\sup_{t\in T} \sum_{i=1}^d t_i x(i) \gs L \E \sup_{t\in T} \sum_{i=1}^d t_i X_i \right\} \subset \bigcup_{(x_\ast,W)\in \ccG} \{x\in \R^d_0:\; \forall_{i\in W} x(i)\gs x_\ast(i)\},\label{eq:locbulbasaur}\\
        &\sum_{(x_\ast,W)\in \ccG} \P(\forall_{i\in W} X_i\gs x_\ast(i) )\ls \delta . \nonumber 
    \end{align}
\end{theo}
\noindent Condition \eqref{eq:locbulbasaur} immediately implies  that
\[\left\{\sup_{t\in T} \sum_{i=1}^d t_i X_i \gs L \E \sup_{t\in T} \sum_{i=1}^d t_i X_i \right\} \subset \bigcup_{(x_\ast,W)\in \ccG} \{\forall_{i\in W} X_i\gs x_\ast(i)\}.\]
 Intuitively, we select only those variables that significantly affect the value of the process. Roughly speaking, Theorem \ref{thm:konc} says that we can specify these $X_i$'s (witnesses) which, if appropriately large, make the whole process large (exceeding a multiple of its mean). \\
Theorem \ref{thm:wykl} may look like mere theoretical curiosity. Surprisingly, we were able to use it to obtain a completely new comparison principle for suprema of canonical processes based on variables that are uniformly distributed on generalized Orlicz balls. This is one of the very few comparison principles for random variables that are dependent.  
\begin{theo}[Comparison principle for generalized Orlicz balls]\label{thm:porow}
There exists a numerical constant $L>0$ with the following property.  Let $f_1,\ldots,f_d:\R_0\to \R_0$ be convex, increasing functions such that $f_i(0)=0$. Let 
 \[B:=\{x\in \R^d_0: \sum_{i=1}^d f_i(x(i)) \leq 1 \}.\]
 Let $X=(X_1,\ldots,X_d)$ be a random vector, uniformly distributed on the convex body $B$. Let $X'_1,\ldots,X'_d$ be independent random variables such that for each $i\leq d$, $X_i \stackrel{D}{=} X'_i$. Then for any $T\subset \R^d_0$ we have
 \[\E \sup_{t\in T} \sum_i t_i X_i \leq L \E \sup_{t\in T} \sum_i t_i X'_i.\]
\end{theo}
\begin{rema}
Theorem \ref{thm:porow} remains true under the assumption that the set $T$ is unconditional (i.e. if $(t_1,\ldots,t_d)\in T$ then $(\pm t_1,\ldots, \pm t_d) \in T$), while the functions  $f_1,\ldots,f_d:\R \to \R_0$ are symmetric, and their restrictions to $\R_0$ are convex and increasing. It is sufficient to note that
\[\E \sup_{t\in T} \sum t_i X_i= \E \sup_{t \in T} \sum_i |t_i| |X_i|, \]
and apply Theorem \ref{thm:porow}.
\end{rema}
\noindent Theorem \ref{thm:wykl} is not restricted to i.i.d. variables, so our comparison result is not restricted to permutational invariant bodies. This level of generality is quite surprising, since such a comparison result was unknown even for the $\ell_p$ balls. This shows the usefulness of our Theorem \ref{thm:wykl}. \\
We suspect that Theorem \ref{thm:wykl} can also be used to derive some Sudakov-type bounds for generalized Orlicz balls. Such bounds say that if we have a stochastic process $X_t$ indexed by the set $T$ containing 'sufficiently many' elements that are 'well' separated then $\E \sup_t X_t$ is 'large'. Optimal bounds are known for canonical processes (i.e. $T\subset \R^d$ and $X_t=\sum t_i X_i$, where $X_1,\ldots,X_d$ are r.v.'s) based on independent r.v.'s, with logarithmically concave tails \cite{Lat1}. In the case of dependent random variables only suboptimal results are known, e.g. \cite[Theorem 4.3]{Latzal}. Optimal bounds are even unknown if $X_t=\sum t_i X_i$ and $X=(X_1,\ldots,X_d)$ is uniformly distributed on the $\ell_p$ ball.
The first premise that Theorem \ref{thm:wykl}  may imply Sudakov-type bounds is that we were able to derive Theorem \ref{thm:porow} from it. The second one is that we deduced a Sudakov-type bound in the 'toy case' of an independent exponential random variable, see Section \ref{sec:app}.\\
Under additional assumptions on the random variables, we can construct a more natural cover in terms of linear functionals. In the theorem below, $\langle \cdot , \cdot \rangle$ stands for the standard scalar product in $\R^d$.
\begin{theo}\label{thm:logwklesle}
   There exists a numerical constant $L>0$ with the following property. Suppose $X_1,\ldots,X_n$ are non-negative i.i.d. random variables such that the function $N(t):=-\ln \P(X_i\geq t)$ is convex.
    Then, there exists a countable set $\ccG\subset\R^d$ for which
    \begin{align*}
        &\left\{x\in \R^d: \sup_{t\in T} \sum_{i=1}^d x(i)t_i \gs L \E \sup_{t\in T} \sum_{i=1}^d X_i t_i\right\} \subset \bigcup_{w\in \ccG} \left\{x\in \R^d: \langle x, w\rangle \gs 1 \right\},\\
      &  \sum_{w\in \ccG} \P\left( \langle (X_1,\ldots,X_d) , w \rangle \gs 1 \right) \leq 1/2.
    \end{align*}
  \end{theo}
  \noindent
 We don't know if Theorem \ref{thm:logwklesle} is true without additional assumptions about the distribution. In particular, such a theorem cannot be deduced from Theorem \ref{thm:wykl}. Surely, for $W\subset [d]$, $x_\ast\in \R^d_0$ 
\[ \{x\in \R^d_0:\; \forall_{i\in W} x(i)\gs x_\ast(i)\} \subset \left\{x\in \R^d_0: \sum_{i\in W} x(i)  \gs \sum_{i\in W}  x_\ast(i) \right\},\]
but it may happen that
\[ \P\left(\sum_{i\in W} X_i \gs \sum_{i\in W} x_\ast(i)\right)\gg \P(\forall_{i\in W} X_i\gs x_\ast(i) ).\]
For example, if $X_1,X_2$ are independent $U[0,1]$ random variables. Then for small $\va>0$
\[\P(X_1\gs 1-\va,X_2 \gs 3/4) \ls \va \ll \P(X_1+X_2 \gs 7/4) \ls \P(X_1+X_2 \gs 7/4-\va).\]

\section{Application and Proofs}\label{sec:app}
We show that Theorems \ref{thm:wykl} and \ref{thm:konc} can provide lower bounds on the expectations of the suprema of stochastic processes. The following fact is part of Talagrand's proof of the Sudakov-type bounds for the exponential process (see  \cite[Proposition 4.3]{TG2}). It is well known and can be proved  by  standard arguments from probability theory. However, our goal is to show that small covers are suitable for deriving Sudakov-type bounds. 
\begin{fak}
    Let $X_1,X_2,\ldots$ be nonnegative i.i.d. r.v.'s with CDF given by $\P(X_1 \gs t)=exp(-t)$. Then 
    \[\E \sup_{|I|\subset [d], |I|= k} \sum_{i\in I} X_i \geq \frac{k}{C}\ln(d/(ek)).\]
\end{fak}
\begin{proof}
Let $p>0$ be defined by
    \begin{equation}\label{eq:wybp}
      e^{p} = \left( \frac{d}{ek}\right)^{k} .
      \end{equation}
      Assume that there is a family $\ccG$ of subsets of $\R^d_0 \times [d]$ such that
      \begin{align}
          \ccF:=\{x\in \R^d_0 \sup_{|I|\subset [d], |I|= k} \sum_{i\in I} x(i) \gs p  \} \subset \bigcup_{(x_\ast,W) \in \ccG} \{ x\in \R^d_0: \forall_{i\in W} x(i) \gs x_\ast(i)>0\}, \label{eq:zaw}
      \end{align}
      To shorten the notation we define $\langle (x_\ast,W) \rangle:=\{ x\in \R^d_0: \forall_{i\in W} x(i) \gs x_\ast(i)>0\}$. Let $I\subset [d], \ |I|=k$ and $e_1,\ldots,e_d$ be standard base of $\R^d$. Then trivially 
      \[x_I:=p/k \sum_{i\in I} e_i \in \ccF\subset \bigcup_{(x_\ast,W)\in \ccG} \langle (x_\ast,W) \rangle.\]
We are only interested in those elements of $\ccG$ that covers $\{x_I: I\subset [d], |I|=k\}$. We define
\begin{align*}
   \ccG'&:=\{(x_\ast,W) \in \ccG: \textrm{exists }I\subset [d],\ |I|=k \textrm{ such that }\ x_I\in \langle (x_\ast,W) \rangle \}, \\
   \ccG'_l&:=\{(x_\ast,W) \in \ccG': |W|=l\}.
\end{align*}
In particular, if $(x_\ast,W) \in \ccG'$ then for some $I\subset [d]$ and $|I|=k$ we have
\begin{equation}\label{eq:osznormx}
 \forall_{i \in W} 0<x_\ast(i) \leq x_I(i) \ls p/k.   
\end{equation}
Also if $(x_\ast,W)\in \ccG'$ then $x_I \in \langle (x_\ast,W) \rangle$ if and only if $W\subset I$. In particular $\ccG'_l=\emptyset$ for $l>k$ and
\begin{align*}
\binom{d}{k}&=|\{I\subset [d]: |I|=k\}|\ls \sum_{l=0}^k \sum_{(x_\ast,W) \in \ccG'_l} |\{I\subset [d]: |I|=k,\ x_I \in \langle (x_\ast,W) \rangle  \}|  \\
&=\sum_{l=0}^k \sum_{(x_\ast,W) \in \ccG'_l} |\{I\subset [d]: |I|=k,\ W \subset I \}|=\sum_{l=0}^k |\ccG'_l|\binom{d-l}{k-l}. 
\end{align*}
We divide the above inequality by the left-hand side and then apply Sauer's Lemma to get that
\begin{align*}
    1 \ls \sum_{l=0}^k |\ccG'_l|\frac{\binom{d-l}{k-l}}{\binom{d}{k}}= \sum_{l=0}^k |\ccG'_l|\frac{\binom{k}{l}}{\binom{d}{l}}\ls \sum_{l=0}^k |\ccG'_l| \left(\frac{ek}{d} \right)^l = \sum_{l=0}^k |\ccG'_l| e^{-p\frac{l}{k}},
\end{align*}
where we recall \eqref{eq:wybp}.On the other hand, since $\ccG'\subset \ccG$, by  \eqref{eq:osznormx}
\begin{align*}
\sum_{(x_\ast,W)\in \ccG} \P(\forall_{i \in W} X_i \gs x_\ast(i))&=\sum_{(x_\ast,W)\in \ccG} e^{-\sum_{i\in W}x_\ast(i)}\gs \sum_{(x_\ast,W)\in \ccG'} e^{-\sum_{i\in W}x_\ast(i)} \\
&\gs \sum_{(x_\ast,W)\in \ccG'} e^{-\frac{p}{k} |W|} =  \sum_{l=0}^k |\ccG'_l|e^{-p\frac{l}{k}}\gs 1.  
\end{align*}
Thus, whenever $\ccG$  satisfies \eqref{eq:zaw}, it also satisfies the above. So there is no small cover of $\ccF$ (see \eqref{eq:zaw}). So Theorem \ref{thm:konc} implies (exponential r.v.'s satisfy \eqref{eq:koncentracjaXniedyskretny} with $C=e$) that 
\[\E \sup_{I\subset[d], |I|=k}\sum_{i\in I} X_i \gs   \frac{p}{C}.\]
\end{proof}
\noindent Now we will prove Theorem \ref{thm:porow}. We begin with some technical facts. 

\begin{fak}\label{fact:logwklesle}
    Let $X=(X_1,\ldots,X_d)$ be a random vector uniformly distributed on a convex body $B$. Define $N_i(t):=-\ln \P(X_i \geq t)$, $s_i:=\inf\{x(i): (x_1,\ldots,x_d)\in B\}$ and $m_i:=\sup \{x(i): (x_1,\ldots,x_d)\in B\}$. Then for each $i\leq d$ the function $N_i$ is strictly increasing and convex on $[s_i,m_i)$. Also $N_i(s_i)=0$ and $\lim_{x \to m_i} N_i(x)=\infty$.
\end{fak}
\begin{proof}
The only non-trivial part of the assertion is the convexity of $N_i$ (in particular $N_i$ is strictly increasing since $B$ is convex). Observe that for any $\lambda \in [0,1]$ and $t,s \in \R$ we have 
    \[ \lambda \{x\in \R^d: x(i) \geq t\} + (1-\lambda) \{x\in \R^d: x(i) \geq s\}= \{x\in \R^d: x(i) \geq \lambda t+ (1-\lambda) s\},  \]
    where in the above we have Minkowski's addition. The measure $\mu$ is  log-concave (see \cite[Example 2.1.3]{iso}) thus,
    \[\P(X_i \geq \lambda t +(1-\lambda)s) \geq \P(X_i\geq t)^\lambda \P(X_i\geq s)^{1-\lambda}. \]
    It is enough to take $\ln$ on both sides and recall the definition of $N_i$.
\end{proof}

\begin{lema}\label{lem:mom_por}
    Assume that the vector $(X_1,\ldots,X_d)$ has log-concave distribution. Then for any semi-norm $\lv \cdot \rv$ on $\R^d$ the following inequality holds
    \[\sqrt{\E \lv (X_1,\ldots,X_d) \rv^2 } \leq \beta_2 \E \lv (X_1,\ldots,X_d) \rv.\]
\end{lema}
\begin{proof}
   This is a simple consequence of \cite[Theorem 1]{latmomlog} and Jensen's inequality.
\end{proof}
\noindent The pivotal property of the uniform distributions on generalized Orlicz balls is that they are negatively associated. For the  convenience of the reader, the definition is given below.
\begin{defi}\label{def:na}[see \cite[Section 1.1]{Wojt} ]
    A random vector $(X_1,\ldots,X_d)$ is said to be negatively associated if for any coordinate-wise increasing, bounded functions $f,g$ and disjoint sets $\{i_1,\ldots,i_k\},\{j_1,\ldots,j_l\} \subset [d]$ we have
\[\mathrm{Cov}(f(X_{i_1},\ldots,X_{i_k}\}),g(X_{j_1},\ldots,X_{j_l})) \leq 0.\]
\end{defi}
\begin{rema}\label{rem:oszprawd}
    If a random vector $(X_1,\ldots,X_d)$ is negatively associated  then for any $\{i_1,\ldots,i_k\}\subset [d]$ and any $x_1,\ldots,x_l \in \R$
    \[\P(X_{i_1} \geq x_1, \ldots, X_{i_k} \geq x_k) \leq \P(X_{i_1} \geq x_1) \cdots \P(X_{i_k} \geq x_k).  \]
    In particular, if $X'_1,\ldots,X'_d$ are independent random variables such that for any $i\leq d$, $X_i \stackrel{D}= X'_i$, then
    \[\P(X_{i_1} \geq x_1, \ldots, X_{i_k} \geq x_k)\leq \P(X'_{i_1} \geq x_1, \ldots, X'_{i_k} \geq x_k).  \]
\end{rema}

\begin{proof}[Proof of Theorem \ref{thm:porow}]
Fix $i\leq d$ and consider $N_i(t)=-\ln \P(X_i \geq t)$ and $s_i,m_i$ as in Fact \ref{fact:logwklesle}. In our case $s_i=0$.  Let $F_i(x):[0,\infty) \to [0,m_i)$ be the inverse of the function $N_i$. By Fact \ref{fact:logwklesle} every $F_i$ is strictly increasing, concave (since $N_i$ is convex), and $F_i(0)=0$. In particular
\begin{equation}\label{eq:bijekcja}
 (x_1,\ldots,x_d) \mapsto (F_1(x_1),\ldots,F_d(x_d))   
\end{equation}
is a bijection  between $[0,\infty)^d$ and $[0,m_1) \times \cdots \times [0,m_d)$, and (by concavity)
\begin{equation}\label{eq:sprawdzwar}
   \forall_{i\leq d} \forall_{x>0} F_i(2x)=F_i(2x)+F_i(0) \leq 2F_i(x).
\end{equation}
Let $\ccE_1,\ldots,\ccE_d$ be i.i.d. non-negative exponential r.v.'s ($\ccE_i$ has the density $g(x)=\1_{x>0} e^{-x}$). Simple computations show that ($N_i(t)=-\ln \P(X'_i \geq t)$ since $X_i \stackrel{D}{=} X'_i$)
\[(F_1(\ccE_1),\ldots,F_d(\ccE_d)) \stackrel{D}{=} (X'_1,\ldots,X'_d).\]
For each $i\leq d$ the function $F_i$ is non-negative and increasing so we have
\[\E F_i(\ccE_i) \geq \P(\ccE_i \geq 1) F_i(1)= e^{-1} F_i(\E \ccE_i). \]
Thus, \eqref{eq:oczekdol} holds with $C'=e^{-1}$ and \eqref{eq:efrosnie} holds with $C'=2\leq e$ (recall \eqref{eq:sprawdzwar}). Trivially $\ccE_1,\ldots,\ccE_d$ satisfy condition \eqref{eq:koncentracjaXniedyskretny} with $C=e$. Thus, the assumptions of Theorem \ref{thm:wykl} are satisfied so there exists a family $\ccG \subset \{(x_\ast,W): x_\ast\in \R^d_0,\ W\subset [d] \}$ for which
\begin{align}
        &\ccF:=\left\{x\in \R^d_0:\sup_{t\in T} \sum_{i=1}^d t_i F_i(x(i)) \gs L \E \sup_{t\in T} \sum_{i=1}^d t_i F_i(\ccE_i) \right\} \subset \bigcup_{(x_\ast,W)\in \ccG} \{x\in \R^d_0:\; \forall_{i\in W} F_i(x(i))\gs x_\ast(i)\},\label{eq:lokal}\\
     & \sum_{(x_\ast,W)\in \ccG} \P(\forall_{i\in W} F_i(\ccE_i)\gs x_\ast(i) )\ls \delta:=\frac{1}{4\beta_2} \label{eq:lokal1},
    \end{align}
    where $\beta_2$ is the constant from Lemma \ref{lem:mom_por}. By \cite[Theorem 1.2]{Wojt} the vector $(X_1,\ldots,X_d)$ is negatively associated (see Definition \ref{def:na}). So by Remark \ref{rem:oszprawd} and \eqref{eq:lokal1}
    \begin{align} \nonumber
\sum_{(x_\ast,W)\in \ccG} \P(\forall_{i\in W} X_i\gs x_\ast(i) ) &\leq \sum_{(x_\ast,W)\in \ccG} \P(\forall_{i\in W} X'_i\gs x_\ast(i) )\\
&=\sum_{(x_\ast,W)\in \ccG} \P(\forall_{i\in W} F_i(\ccE_i)\gs x_\ast(i) )  \leq \frac{1}{4\beta_2}.\label{eq:chopin1}
         \end{align}
    Since the mapping \eqref{eq:bijekcja} is a bijection,   formula \eqref{eq:lokal} implies that
         \begin{multline}\label{eq:chopin2}
\left\{x\in [0,m_1) \times \cdots \times [0,m_d) :\sup_{t\in T} \sum_{i=1}^d t_i x(i) \gs L \E \sup_{t\in T} \sum_{i=1}^d t_i F_i(\ccE_i) \right\} \\
\subset \bigcup_{(x_\ast,W)\in \ccG} \{x\in [0,m_1) \times \cdots \times [0,m_d):\; \forall_{i\in W}x(i) \gs x_\ast(i)\}.
         \end{multline}
         By \eqref{eq:chopin1} and \eqref{eq:chopin2} we get (recall that $(F_1(\ccE_1),\ldots,F_d(\ccE_d))\stackrel{D}{=}(X'_1,\ldots,X'_d)$)
         \[\P\left(\sup_{t\in T} \sum_{i=1}^d t_i X_i \gs L \E \sup_{t\in T} \sum_{i=1}^d t_i X'_i\right)\leq \frac{1}{4\beta_2}. \]
         On the other hand, by Paley-Zygmund's inequality and Lemma \ref{lem:mom_por}
         \[\P\left(\sup_{t\in T} \sum_{i=1}^d t_i X_i \gs \frac{1}{2}  \E \sup_{t\in T} \sum_{i=1}^d t_i X_i\right) \gs \frac{1}{4} \frac{\left(\E \sup_{t\in T} \sum_{i=1}^d t_i X_i\right)^2 }{\E \sup_{t\in T} \left(\sum_{i=1}^d t_i X_i\right)^2} \geq  \frac{1}{4\beta_2}.\]
        The above computations yield
         \[ L \E \sup_{t\in T} \sum_{i=1}^d t_i X'_i \geq \frac{1}{2} \E \sup_{t\in T} \sum_{i=1}^d t_i X_i.\]
\end{proof}

\noindent We conclude this section with the proof of Theorem \ref{thm:logwklesle}. We begin with a technical lemma.
\begin{lema}\label{lem:formulamomenty}
Suppose $N(t):[0,\infty]\to [0,\infty]$ is convex function and $N(0)=0, N(1)=1$. Let $x_1,\ldots,x_d \gs 1$ and consider $p=\sum_{i\leq d} N(x(i))$. Then
\[\sup\left\{\sum_{i\leq d} \frac{N(x(i))}{x(i)} (1+b_i) \mid \sum_{i\leq d} N(b_i) \leq p \right\} \leq 3p.\]
\end{lema}
\begin{proof}
 Fix $b_1,\ldots,b_d \in \R^d_0$ such that $\sum_{i\leq d} N(b_i)\leq p$. Define $N^\ast(t)=\sup_{s\geq 0} st-N(s)$. Since $x_1\ldots,x_d \gs 1$
 \begin{equation}\label{eq:loc98}
  \sum_{i\leq d} \frac{N(x(i))}{x(i)} (1+b_i) \leq \sum_{i\leq d} N(x(i)) +\sum_{i\leq d} N^\ast \left(\frac{N(x(i))}{x(i)} \right) +\sum_{i\leq d} N(b_i) \leq \sum_{i\leq d} N^\ast \left(\frac{N(x(i))}{x(i)} \right)+2p.    
 \end{equation}
By definition
 \[ N^\ast \left(\frac{N(x(i))}{x(i)} \right)=\sup_{s\gs 0} s\left(\frac{N(x(i))}{x(i)}-\frac{N(s)}{s}\right)=:\sup_{s\gs 0} s F(s). \]
 Due to the convexity of $N$, the function $F(s)$ is non-positive for $s>x(i)$ and non-negative for $s \leq x(i)$.  Thus the supremum is attained for some $s(x(i))\leq x(i)$. As a result
 \[N^\ast \left(\frac{N(x(i))}{x(i)} \right)=\frac{N(x(i))}{x(i)}\cdot s(x(i))-N(s(x(i))) \ls N(x(i)).\]
 Plugging the above into \eqref{eq:loc98} yields the assertion.
\end{proof}

   \begin{proof}[Proof of Theorem \ref{thm:logwklesle}]
   W.l.o.g. we may assume that $N(1)=1$. In particular $\E X_i \gs \P(X_i \gs 1)=1/e$.  Since $N$ is convex and $N(0)=0,\ N(1)=1$ we get for $t\gs 1$ that  $N(2t)\gs 2N(t)\gs N(t)+1$. Thus, \eqref{eq:koncentracjaXniedyskretny} holds with $C=e$ and Theorem \ref{thm:konc}, with $\delta=1/2$ and $K=e$ (so that $K\E X_i=1$), implies that there exists set $\ccG\subset \R^d\times 2^{[d]}$ such that
   \begin{align}
       &\left\{x\in \R^d_0:\sup_{t\in T} \sum_{i=1}^d x(i)t_i \gs L \E \sup_{t\in T} \sum_{i=1}^d t_i X_i\right\} \subset \bigcup_{(x_\ast,W)\in \ccG}\{x\in \R^d_0: \forall_{i\in W} x(i) \gs x_\ast(i)\}\label{eq:locpokemon},\\
       \label{wysumuj}&\sum_{(x_\ast,W)\in \ccG} \P(\forall_{i\in W}X_i \gs x_\ast(i)) \ls 1/2,\\
    \label{eq:oszcdolu}   &\forall_{(x_\ast,W)\in \ccG} \forall_{i\in W} x_\ast(i)\gs 1.
   \end{align}
Let $L'$ be a large constant to be chosen later. Trivially, from \eqref{eq:locpokemon}
   \begin{align*}\nonumber
    \left\{x\in \R^d_0:\sup_{t\in T} \sum_{i=1}^d x(i)t_i \gs LL' \E \sup_{t\in T} \sum_{i=1}^d t_i X_i\right\} &\subset \bigcup_{(x_\ast,W)\in \ccG}\{x\in \R^d_0: \forall_{i\in W} x(i) \gs L' x_\ast(i)\}\\
  &\subset  \bigcup_{(x_\ast,W)\in \ccG}\left\{x\in \R^d_0: \sum_{i\in W} \frac{N(x_\ast(i))}{x_\ast(i)}x(i) \geq L' p(x_\ast,W)\right\},  
   \end{align*}
where 
\begin{equation*}
 p(x_\ast,W):=\sum_{i\in W} N(x_\ast(i))=-\ln \P(\forall_{i\in W}X_i \gs x_\ast(i)).   
\end{equation*}
It remains to show that 
\[\P\left(\sum_{i\in W} \frac{N(x_\ast(i))}{x_\ast(i)}X_i\gs L'p(x_\ast,W)\right) \leq e^{-p(x_\ast,W)}=\P(\forall_{i\in W}X_i \gs x_\ast(i)),\]
and use \eqref{wysumuj}. From \eqref{eq:oszcdolu}
   \[p(x_\ast,W)=\sum_{i\in W} N(x_\ast(i)) \gs |W|N(1)=|W|\gs 1.\]
By  \cite[Theorem 2.1]{latloch} (cf. formula $(3.1)$ therein) and Lemma \ref{lem:formulamomenty} (from \eqref{eq:oszcdolu}, $\forall_{i \in W}$ $x_\ast(i) \geq 1$)
   \[\lv \sum_{i\in W}\frac{N(x_\ast(i))}{x_\ast(i)}X_i \rv_{p(x_\ast,W)} \leq \kappa \sup\left\{\sum_{i\in W} \frac{N(x_\ast(i))}{x_\ast(i)}(1+b_i) \mid \sum_{i\in W} N(b_i)\leq p(x_\ast,W) \right\}\ls 3\kappa  p(x_\ast,W),\]
where $\kappa>0 $ is a numerical constant. So Markow's inequality yields for $L'=3e\kappa$
\begin{align*}
\P\left(\sum_{i\in W}  \frac{N(x_\ast(i))}{x_\ast(i)}X_i\gs L'p(x_\ast,W)\right) &\leq \left( \frac{\lv \sum_{i\in W}  \frac{N(x_\ast(i))}{x_\ast(i)}X_i \rv_{p(x_\ast,W)}}{L'p(x_\ast,W)} \right)^{p(x_\ast,W)}\\
&\ls e^{-p(x_\ast,W)}.    
\end{align*}

  \end{proof}

\section{Proof of Theorem  \ref{thm:wykl}}\label{sec:pmain}
We start with a language change. Let $X_1,\ldots,X_d$ be i.i.d. random variables taking  values in $[n]=\{1,2,\ldots,n\}$ and strictly increasing functions $f_i:[n]\to \R_0$ such that $f_i(1)=0$ for each $i\leq d$. To shorten the notation we set
\[S(T):=\E \sup_{t\in T} \sum_{i=1}^d t_i f_i(X_i).\]
For $x\in [n]^d$ we define $t^{x,f}\in T$ as
\[\sum_{i=1}^d t^{x,f}_if_i(x(i))=\sup_{t\in T} \sum_{i=1}^d t_if_i(x(i)).  \]
When it is not confusing, we will abbreviate $t^{x,f}$ to $t^x$.  By definition,
\[S(T)=\E \sum_{i=1}^d t^X_i f_i(X_i)=\E \sup_{x\in [n]^d} \sum_{i=1}^d t^x(i)f_i(X_i).\]

\begin{defi}\label{def:cover}
   Let $L>0$ and 
   \begin{equation}\label{eq:deff}
      \ccF:=\left\{x\in [n]^d: \sup_{t\in T} \sum_{i=1}^d t_i f_i(x(i)) \gs L \right\}. 
   \end{equation}
 We say that  $\ccG\subset \{(x_\ast,W):x_\ast \in \{2,\ldots,n\}^d, W\subset [d]\}$  is a  cover of $\ccF$ if for any $x\in \ccF$ there exists $(x_\ast,W)\in \ccG$ such that for all $i\in W$ $x(i)\gs x_\ast(i)$. Equivalently $\forall_{i\in W}$ $f_i(x(i))\gs f_i(x_\ast(i))>0$ (for any $i\leq d$, $f_i$ is strictly increasing and $f_i(1)=0$). In other words
       \[\ccF\subset \bigcup_{(x_\ast,W)\in \ccG} \{x\in [n]^d: \; \forall_{i\in W}\; x(i) \gs x_\ast(i)>1\}=\bigcup_{(x_\ast,W)\in \ccG} \{x\in [n]^d: \; \forall_{i\in W}\; f_i(x(i)) \gs f_i(x_\ast(i))>0\}. \]
    \end{defi}
\noindent The reason why we require $x_\ast(i)>1$ ($f_i(x_\ast(i))>0$) is explained after the next definition.
\begin{rema}
 It would be more natural to define a cover $\ccG \subset \{(x_\ast,W): x_\ast\in \R_0, W\subset [d] \}$ as a family that satisfies
 \[\ccF\subset \bigcup_{(x_\ast,W)\in \ccG} \{x\in [n]^d: \; \forall_{i\in W}\; f_i(x(i)) \gs x_\ast(i)\}.\]
 But then we would have to write expressions like $\lfloor f^{-1}(x_\ast(i)) \rfloor$. Thus, we chose a less natural definition that leads to simpler formulas.
\end{rema}

\begin{defi}\label{def:smallfam}
 Let $\delta< 1$. We  say that $\ccF\subset [n]^d$ given by \eqref{eq:deff} is $\delta$-small if there exists a cover $\ccG$ (in the sense of Definition \ref{def:cover}) such that
 \begin{equation}\label{eq:locnidorino}
\sum_{(x_\ast,W)\in \ccG} \P(\forall_{i\in W} \;  X_i= x_\ast(i))=\sum_{(x_\ast,W)\in \ccG} \P(\forall_{i\in W} \;  f_i(X_i)= f_i(x_\ast(i)))\ls \delta,
 \end{equation}
  where the equality holds since for for $i\leq d$, $f_i$ is strictly increasing.
\end{defi}
\noindent  Let $\ccG$ be a cover of $\ccF$ in the sense of the Definition \ref{def:cover}. In particular
\[\ccF\subset \bigcup_{(x_\ast,W)\in \ccG} \{x\in [n]^d: \; \forall_{i\in W}\; x(i) \gs x_\ast(i)>1\}. \]
Assume for a moment that we don't require $x_\ast(i)>1$ for every $(x_\ast,W)\in \ccG$ and every $i\in W$. Fix $(x_\ast,W)\in \ccG$ and define
\[\hat{x}(i)=\begin{cases}
x_\ast(i) &i\in W \\ 1 & i\notin W.
    \end{cases}
\]
Consider $\hat{\ccG}:=\{(\hat{x},[n]):(x_\ast,W)\in \ccG\}$ ($\hat{x}$ is a function of $(x_\ast,W)$). Then trivially
\[\ccF\subset \bigcup_{(\hat{x},[n])\in \hat{\ccG}} \{x\in [n]^d: \forall_{i\in [n]} \  x(i)\gs \hat{x}(i)\}.\]
Also
\[ \sum_{(x_\ast,W)\in \ccG} \P(\forall_{i\in W} \;  X_i= x_\ast(i)) \gs  \sum_{(\hat{x},[n])\in \hat{\ccG}} \P(\forall_{i\in [n]} \  X_i= \hat{x}(i))\]
So $\hat{\ccG}$ would be a better cover of $\ccF$ than $\ccG$. But it doesn't provide any additional information. For this reason, we impose the condition that  $\forall_{(x_\ast,W)\in \ccG}\forall_{i\in W} x_\ast(i)>1$.

\noindent Now we state a discrete version of Theorem \ref{thm:wykl} using the introduced notation.
\begin{theo}\label{thm:mainjezkoncentracja}
Let $\delta\in (0,1)$,\ $K\gs \lceil 12e/\delta \rceil$ be a natural number, $X_1,\ldots,X_d$ be i.i.d. random variables with values in $[n]$ such that $\P(X_i=1) \gs 1-1/K$. Assume that for each $i=1,\ldots,d$ $f_i:[n]\to \R_0$ is a strictly increasing function and $f_i(1)=0$.  Then for any $K'>4K$ the family 
\[\ccF:=\left\{x\in [n]^d: \sup_{t\in T}   \sum^d_{i=1} t_i f_i(x(i)) \gs K'S(T) \right\}\]
is  $\delta$-small. 
\end{theo}
\noindent We have three remarks about Theorem \ref{thm:mainjezkoncentracja}. The first one is that it is enough to prove it for $\hat{X}_i=X_i \1_{X_i \geq K \E X_i}$ instead of $X_i$. So the condition that $\P(X_i=1) \gs 1-1/K$ does not reduce the generality of the theorem.  The argument is used in the proof of Theorem \ref{thm:wykl} below.\\
\noindent Second, Theorem \ref{thm:mainjezkoncentracja} goes  beyond the i.i.d. case. Instead of considering i.i.d. random variables $X_1,\ldots,X_d$, we consider $Y_1,\ldots,Y_d$ where $Y_i=f_i(X_i)$ and $f_1,...,f_d$ are (possibly) different functions. This is still weaker than considering arbitrary, independent non-negative random variables $Y_1,\ldots,Y_d$. For such a  case, our argument fails in exactly one, but a crucial place. It is inequality \eqref{eq:loc2} in which we compute the probability of a certain event. In the general case, the formula would be too complicated to lead to any non-trivial bounds.\\
\noindent Finally, Theorem \ref{thm:mainjezkoncentracja}  implies the Park-Pham result \eqref{eq:selektory1}. Let $X_1,\ldots,X_d$ be i.i.d. Bernoulli random variables with the probability of success equal to $p>0$. Theorem \ref{thm:mainjezkoncentracja} implies that there exists a $1/2$-small cover  $\ccG\subset \{(x_\ast,W): x_\ast\equiv 1, W\subset [d] \}$  of
\[\ccF=\{x\in \{0,1\}^d: \sup_{t\in T} \sum_{i=1}^d t_i \delta_i \gs K' \delta(T) \}\]
(by a small abuse of notation we consider $x\in \{0,1\}^d$ instead of  $x\in \{1,2\}^d$).
By definition of a $1/2$-small cover
\begin{align*}
 \ccF&\subset \bigcup_{(x_\ast,W)\in \ccG} \{x\in \{0,1\}^d: \forall_{i\in W} x(i)=1 \},\\
1/2 &\gs  \sum_{(x_\ast,W)\in \ccG} \P( \forall_{i\in W}\  X_i=1)= \sum_{(x_\ast,W)\in \ccG} p^{|W|},
\end{align*}
which is exactly the result of Park and Pham.\\
\noindent Theorem \ref{thm:mainjezkoncentracja} is a consequence of the following result.
\begin{theo}\label{thm:rozdconentracja}
Fix the natural number $K \gs\lceil 12e/\delta\rceil$. Let  $X_1,\ldots,X_d$ be i.i.d. random variables with values in $[n]$ such that $\P(X_i=1) \gs 1-1/K$. Assume that for each $1\ls i \ls d$,  $f_i:[n]\to \R_0$ is a strictly increasing function and $f_i(1)=0$.
Let $L>0$ and define
\[\ccF:=\cbr{x\in [n]^d:   \sum^d_{i=1} t^x(i) f_i(x(i)) >L}.\]
Assume that $\ccF$ is not $\delta$-small. Let $(Y^l_i)_{i\ls d, l\ls  K}$ be independent copies of $X_1$.  Then
\[\E \sup_{x\in \ccF}\sum_{i\ls d} \sum_{l\ls K} t^x(i) f_i(Y^l_i)\gs L/4.  \]
\end{theo}

\begin{proof}[Proof of Theorem \ref{thm:mainjezkoncentracja}.]

Let $K'>0$ and define
\begin{align*}
  \ccF:=\cbr{x\in [n]^d:   \sum^d_{i=1} t^x(i) f_i(x(i)) >K' S(T)}.
\end{align*}
By trivial argument
\begin{align}
  \E \sup_{x\in \ccF}\sum_{i\ls d} \sum_{l\ls K} t^x(i)f_i(Y^l_i)\ls \sum_{l\ls K} \E \sup_{t\in T} \sum_{i\ls d} t_i f_i(Y^l_i)=KS(T).\label{eq:loc91}
\end{align}
If $\ccF$ is not $\delta$-small then by Theorem \ref{thm:rozdconentracja}
\[  \E \sup_{x\in \ccF}\sum_{i\ls d} \sum_{l\ls K} t^x(i)f_i(Y^l_i) \geq K'/4 S(T).\]
So $\ccF$ is $\delta$ small for $K'>4K$.
\end{proof}

\noindent We deduce Theorem \ref{thm:wykl} from Theorem \ref{thm:mainjezkoncentracja}. But first we show a technical lemma which allows to construct better covers using our tail condition.
\begin{lema}\label{lem:poprodz}
Under the assumption of Theorem \ref{thm:wykl} the following holds. Let $\ccG \subset \{ (x_\ast,W):W\subset [d],x_\ast \in \R_0^d\}$ be such that 
\begin{align}
&\forall_{(x_\ast,W)\in \ccG} \forall_{i\in W} x_\ast(i) \geq K \E X_i, \nonumber \\
     &\left\{x\in \R^d_0:\sup_{t\in T} \sum_{i=1}^d t_i f_i(x(i)) \gs L \E \sup_{t\in T} \sum_{i=1}^d t_i f_i(X_i) \right\} \subset \bigcup_{(x_\ast,W)\in \ccG} \{x\in \R^d_0:\; \forall_{i\in W} f_i(x(i))\gs f_i(x_\ast(i))\} \label{eq:locaa1}, \\
     &\sum_{(x_\ast,W)\in \ccG} \P\big(\forall_{\in W}  f_i(X_i) \in [f_i(x_\ast(i),f_i(2x_\ast(i)))\big) \leq \delta. \nonumber
\end{align}
Then there exists $\beta>1$ depending only on $C,C'$ (constants appearing in conditions \eqref{eq:koncentracjaXniedyskretny},\eqref{eq:efrosnie}) such that
\begin{equation}\label{eq:locaa2}
 \left\{x\in \R^d_0:\sup_{t\in T} \sum_{i=1}^d t_i f_i(x(i)) \gs \beta L \E \sup_{t\in T} \sum_{i=1}^d t_i f_i(X_i) \right\} \subset \bigcup_{(x_\ast,W)\in \ccG} \{x\in \R^d_0:\; \forall_{i\in W} f_i(x(i))\gs  \beta f_i( x_\ast(i))\}    
\end{equation}
and
\begin{equation*}
 \sum_{(x_\ast,W)\in \ccG} \P\big(\forall_i   f_i(X_i) \geq \beta f_i(  x_\ast(i)\big) \leq \delta.   
\end{equation*}
\end{lema}
\begin{proof}
Formula \eqref{eq:locaa2} follows from \eqref{eq:locaa1} and homogeneity for any $\beta>1$. By \eqref{eq:koncentracjaXniedyskretny} we get that for any $(x_\ast,W)\in \ccG$ and for any $i \in W $ (since then $x_\ast(i)\geq K \E X_i$)
\[\P(X_i\in [x_\ast(i),2x_\ast(i)])\gs (1-1/C) \P(X_i\gs x_\ast(i)) \geq (1-1/C) C^\rho \P(X_i\gs 2^\rho x_\ast(i)) .\]
We pick the smallest $\rho \in \mathbb{N}$ satisfying $(1-1/C) C^\rho \geq 1$, so that the above and our assumptions give
\begin{align*}
 \delta &\geq   \sum_{(x_\ast,W)\in \ccG} \P\big(\forall_{i\in W} X_i \in [x_\ast(i),2x_\ast(i))\big) \geq \sum_{(x_\ast,W)\in \ccG} P(\forall_{i\in W} X_i\gs 2^\rho x_\ast(i))\\
 &=\sum_{(x_\ast,W)\in \ccG} P(\forall_{i\in W} f_i(X_i)\gs f_i(2^\rho x_\ast(i)))\gs \sum_{(x_\ast,W)\in \ccG} P(\forall_{i\in W} f_i(X_i)\gs (C')^\rho f_i(x_\ast(i))),
\end{align*}
where in the equality we used that each $f_i$ is increasing and in the last inequality we used  \eqref{eq:efrosnie} (applied $\rho$ times). We pick $\beta=(C')^\rho$.
\end{proof}

\begin{proof}[Proof of Theorem \ref{thm:wykl}]
W.l.o.g. we may assume that $\E X_i=1$ (otherwise we replace $X_i$ by $X_i/(\E X_i)$ and $f_i(t)$ by $\hat{f}_i(t)=f_i(t \E X_i)$). We pick the smallest $l\in \N$ for which
\begin{equation}\label{eq:wybl}
 2^l \gs K ,    
\end{equation}
and the smallest natural $n>l$ for which
\begin{align}\label{eq:wyben}
    \sum_{i=1}^d\P(f_i(X_i) \geq f_i(2^n)) \leq \delta/2.
\end{align}
By Jensen's inequality and \eqref{eq:oczekdol}
 \[\E \sup_{t\in T} \sum_{i=1}^d t_i f_i(X_i) \gs   \sup_{t\in T} \sum_{i=1}^d t_i \E f_i(X_i) \gs (C')^{-1} \sup_{t\in T} \sum_{i=1}^d t_i f_i(1).\]
Using that $f_1,\ldots,f_d$ are increasing,  \eqref{eq:efrosnie}, and the above
\[ \sup_{t\in T} \sum_{i=1}^d t_i f_i(x(i)) \1_{x(i) \ls 2^l } \leq  \sup_{t\in T} \sum_{i=1}^d t_i f_i(2^l ) \leq (C')^l  \sup_{t\in T} \sum_{i=1}^d t_i f_i(1) \leq (C')^{l+1} \E \sup_{t\in T} \sum_{i=1}^dt_i f_i(X_i) .   \]
This implies that for sufficiently large $L$ (we recall that $l$ depends only on $K$) and any $x\in \ccF$ (which is defined in \eqref{eq:defef})
 \begin{equation}\nonumber
 \sup_{t\in T} \sum_{i=1}^d t_i f_i(x(i)) \1_{x(i) \gs 2^l} \gs (L-(C')^{l+1}) \E \sup_{t\in T} \sum_{i=1}^dt_i f_i(X_i) \gs  \frac{L}{2} \E \sup_{t\in T} \sum_{i=1}^dt_i f_i(X_i) \1_{X_i \in [2^l, 2^n)}.  
 \end{equation}
So, clearly
\begin{align}\label{eq:loc981}
\ccF  \subset \left\{x\in \R^d_0:\ \sup_{t\in T} \sum_{i=1}^d t_i f_i(x(i))\1_{x(i) \in [2^l, 2^n)} \gs  \frac{L}{2} \E \sup_{t\in T} \sum_{i=1}^d t_i f_i(X_i)\1_{X_i \in [2^l, 2^n)} \right\} \cup \bigcup_{i=1}^d \{x\in \R^d_0: x(i) \gs 2^n \}. \nonumber
\end{align}
 Let  $g(t):=\sum_{k=l}^{n-1} 2^k \1_{2^k \ls t < 2^{k+1}}$. From \eqref{eq:efrosnie} we get that
\[\forall_{t\in [2^l,2^n)}\\ f_i(g(t)) \ls f_i(t)  \ls  C' f_i(g(t)).\] 
As a consequence (equality follows because $f_i$ are strictly increasing)
\begin{align}\label{eq:col1}
\ccF \subset \ccF_g \cup \bigcup_{i=1}^d \{x\in \R^d_0: x(i) \gs 2^n \}=\ccF_g \cup \bigcup_{i=1}^d \{x\in \R^d_0: f_i(x(i)) \gs f_i(2^n) \},    
\end{align}
where
\[\ccF_g=\left\{x\in [0,2^n)^d :\ \sup_{t\in T} \sum_{i=1}^d t_i f_i(g(x(i)))\1_{x(i) \in [2^l, 2^n)}  \gs \frac{L}{2C'} \E \sup_{t\in T} \sum_{i=1}^d t_i f_i(g(X_i)) \1_{x(i) \in [2^l, 2^n)}  \right\}.\]
\noindent Consider 
\[ X'_i=\begin{cases} 1 &X_i\in [0,2^l)\cup [2^n,\infty) \\ \log_2g(X_i)-l+2 &X_i\in [2^l,2^n)
\end{cases}, f'_i(k)=\begin{cases}
    0 & k=1 \\ f_i(2^{k+l-2}) &k=2,\ldots,n+1-l.
\end{cases} 
\]
It is easy to check that 
\begin{equation}\label{eq:rownosc}
 f_i(g(X_i))\1_{x(i) \in [2^l, 2^n)} =f'_i(X'_i) .  
\end{equation}
In particular
\begin{equation}\label{eq:exprowna}
    \E \sup_{t\in T} \sum_i t_i f'_i(X'_i)=\E \sup_{t\in T} \sum_i t_i f_i(g(X_i))\1_{X_i \in [2^l, 2^n)} .
\end{equation}
Define
\[\ccF'=\left\{x'\in [n+1-l]^d: \sup_{t\in T} \sum_{i=1}^d t_i f'_i(x'(i)) \gs \frac{L}{2C'} \E \sup_{t\in T} \sum_{i=1}^d t_i f'_i(X'_i) \right\}.\]
  By Markov's inequality 
\[\P(f'_i(X'_i)=f'_i(1))=\P(f'_i(X'_i)=0)\gs\P(X_i\in [0,2^l))\gs  1-2^{-l}\gs 1-1/K.\]
Thus, for  $L$ large enough, Theorem \ref{thm:mainjezkoncentracja}  asserts that there exists a cover 
\begin{equation}
    \ccG'\subset \{(x'_\ast,W'):x'_\ast\in \{2,\ldots,n+1-l\}^d, W\subset [d] \}\label{eq:rodzinazawarta},
\end{equation}
such that
\begin{align}
&\ccF' \subset \bigcup_{(x'_\ast,W')\in \ccG'} \{x'\in [n+1-l]^d:\; \forall_{i \in W'} f'_i(x'(i))\gs f'_i(x'_\ast(i)) >0\},\label{eq:zfartem} \\
& \delta/2\gs \sum_{(x'_\ast,W')\in \ccG' } \P (\forall_{i
    \in W'} f'_i(X'_i) =f'_i(x'_\ast(i))=\sum_{(x'_\ast,W')\in \ccG' } \P (\forall_{i
    \in W'} f'_i(X'_i) =f_i(2^{x'_\ast(i)+l-2})).\label{eq:marinesrodzina}
\end{align}
The inequality $f'_i(x'_\ast(i)) >0$ occurring in \eqref{eq:zfartem} holds, since $f'_i$ is strictly increasing (since $f_i$ is), $f'_i(1)=0$ and $x'_\ast(i)\gs 2$ (from \eqref{eq:rodzinazawarta}). The latter also implies formula for $f'_i(x'_\ast(i))$ so that equality in \eqref{eq:marinesrodzina} holds. 
Fix $y\in \ccF_g$ and define
\begin{equation}\label{eq:formulay'}
 y'(i):=\begin{cases} 1 &y(i)\in [0,2^l)\cup[2^n,\infty)\\ \log_2(g(y(i)))-l+2 &y(i)\in[2^{l},2^{n}).  \end{cases}\  
\end{equation}
In particular
\begin{equation}\label{eq:formulay'1}
 \forall_{i \in [d]} f_i(g(y(i)))\1_{y_i \in [2^l, 2^n)} =f'_i(y'(i)),      
\end{equation}
so that
\begin{align*}
\sup_{t\in T} \sum_i t_i f_i(g(y(i)))\1_{y_i \in [2^l, 2^n)} =\sup_{t\in T} \sum_i t_i f'_i(y'(i)).
\end{align*}
The latter equation and \eqref{eq:exprowna} imply that $y'\in \ccF'$ (since $y\in \ccF_g$). From \eqref{eq:zfartem} there exists $(y'_\ast,W')\in \ccG'$ such that (in the equality we use \eqref{eq:formulay'1})
\[\forall_{i\in W'} f_i(g(y(i)))\1_{y_i \in [2^l, 2^n)} =  f'_i(y'(i))\gs f'_i(y'_\ast(i)) >0. \]
Since $y$ was chosen arbitrarily we obtain
\[ \ccF_g \subset \bigcup_{(x'_\ast,W')\in \ccG'} \left\{x\in \R^d_0:\; \forall_{i\in W'} f_i(x(i))\gs f_i(2^{x'_\ast(i)+l-2})  \right\}.  \]
By the above and \eqref{eq:col1}
\begin{align*}
\ccF\subset \bigcup_{(x'_\ast,W')\in \ccG'} \left\{x\in \R^d_0:\; \forall_{i\in W'} f_i(x(i))\gs f_i(2^{x'_\ast(i)+l-2})  \right\} \cup \bigcup_{i=1}^d \{x\in \R^d_0: f_i(x(i)) \gs f_i(2^n) \}.
\end{align*}
Substituting formula \eqref{eq:rownosc} into \eqref{eq:marinesrodzina} and using \eqref{eq:wyben}   yields
\begin{align}
 \delta&\gs \sum_{(x'_\ast,W')\in \ccG' } \P \left(\forall_{i
    \in W'}f_i(g(X_i))\1_{X_i \in [2^l, 2^n)} =f_i(2^{x'_\ast(i)+l-2})\right)+\sum_i \P \left(f_i(X_i) \geq f_i(2^n) \right)\nonumber \\
    &=\sum_{(x'_\ast,W')\in \ccG' } \P \left(\forall_{i
    \in W'} X_i \in \left[ 2^{x'_\ast(i)+l-2},2^{x'_\ast(i)+l-1}\right)\right)+\sum_i \P \left(f_i(X_i) \geq f_i(2^n) \right)\nonumber\\
    &=\sum_{(x'_\ast,W')\in \ccG' } \P \left(\forall_{i
    \in W'} f_i(X_i) \in \left[ f_i(2^{x'_\ast(i)+l-2}),f_i(2^{x'_\ast(i)+l-1})\right)\right)+\sum_i \P \left(f_i(X_i) \geq f_i(2^n) \right), \label{eq:locpryncypia}
\end{align}
where in the last line we used that $f_i$ are strictly increasing. By construction, \eqref{eq:wybl} and our normalization
\[\forall_{(x'_\ast,W')\in \ccG'} \forall_{i\in W'} x'_\ast(i) \geq 2 \Rightarrow 2^{x'_\ast(i)+l-2} \geq  2^l \geq K \E X_i.\]
Thus we conclude the statement by invoking Lemma \ref{lem:poprodz}.
\end{proof}
\begin{rema}\label{rem:zmiendow}
We could assume 
\begin{equation}\label{eq:locogon}
  \forall_{t\gs K\E X_i }\ \P(X_i\gs t )\gs C \P(X_i\gs M t),  
\end{equation}
instead of condition \eqref{eq:koncentracjaXniedyskretny} (in which $M=2$). It is enough to consider $g(t):=\sum_{k} M^k \1_{M^k \leq t < M^{k+1}}$ in the proof above. The only difference (apart from some editorial changes, e.g. to consider $\log_M$ instead of $\log_2$) is that instead of $\eqref{eq:locpryncypia}$ we get
\[\delta \geq \sum_{(x'_\ast,W')\in \ccG' } \P \left(\forall_{i
    \in W'} f_i(X_i) \in \left[ f_i(M^{x'_\ast(i)+l-2}),f_i(M^{x'_\ast(i)+l-1})\right)\right)+\sum_i \P \left(f_i(X_i) \geq f_i(M^n) \right).\]
Then it is enough to prove a version of Lemma \ref{lem:poprodz} which is adapted to the tail condition \eqref{eq:locogon}. We leave the details to the reader.
\end{rema}
\section{Proof of Theorem \ref{thm:rozdconentracja}}\label{sec:pdisc}
In this section, we prove Theorem \ref{thm:rozdconentracja}. For the rest of this work we fix $L>0$ and a functions
\begin{equation}\label{eq:wlasnoscif}
 f_i:[n]\to \R_+\cup\{0\} \textrm{ is strictly increasing, } f_i(1)=0,\ i=1,\ldots,d.
\end{equation}
To shorten the notation, denote
\[\ccF:=\cbr{x\in [n]^d:   \sum^d_{i=1} t^x(i) f_i(x(i)) >L}.\]
We begin with a definition.
\begin{defi}\label{def:bady}
We say that $y=(y^l(i))_{l\ls K,i\ls d}\in [n]^{Kd}$ is bad if  
\[
\sup_{x\in \ccF}\sum^K_{l=1} \sum^d_{i=1}t^x(i) f_i(y^l(i)) \ls L/2.
\]	    
\end{defi}
\noindent Theorem \ref{thm:rozdconentracja} is a rather simple consequence of the following lemma.
\begin{lema}\label{lem:prawdzly}
Let $K \gs \left\lceil 12e/\delta\right\rceil$ be a natural number and let for each $i\leq d$ $f_i:[n]\to \R_0$ be a strictly increasing function, such that $f_i(1)=0$. Let $X_1,\ldots,X_d$ be i.i.d. non-negative random variables for which  $\P(X_1=1)=\P(f_i(X_i)=0)\gs 1-1/K$. Suppose $\ccF$ is not $\delta$-small. Let $Y=(Y^l_i)_{l\ls K,i\ls d}\in [n]^{Kd}$ where $(Y^l_i)_{l\ls K,i\ls d}$ are independent copies of $X_1$. Then
\[\P(Y \textrm{ is bad})\ls \frac{1}{2}.\]
\end{lema}
\begin{proof}[Proof of Theorem \ref{thm:rozdconentracja}]
 Lemma \ref{lem:prawdzly} states that 
 \[\mathrm{Med}\left(\sup_{x\in \ccF}\sum^K_{l=1} \sum^d_{i=1}t^x(i) f_i(Y^l(i))\right)\gs L/2,\]
 so by a trivial argument
 \[\E \sup_{x\in \ccF}\sum^K_{l=1} \sum^d_{i=1}t^x(i) f_i(Y^l(i)) \gs L/4.\]
\end{proof}

\noindent The rest of the section is devoted to the proof of Lemma \ref{lem:prawdzly}. We define
\begin{align*}
  S_y(j)&:=\{i\in [d]:\;   \exists_{l\in [K]} \; y^l(i)\gs j \}=\{i\in [d]:\;   \exists_{l\in [K]} \; f_i(y^l(i))\gs f_i(j) \}\\
 S^{=}_y(j)&:=S_{y}(j)\backslash S_y(j+1)=\{ i\in [d]:\; \forall_{l\in [K]}y^l(i)\ls j,\    \exists_{l\in [K]} \; y^l(i)=j \} \\
 &= \{ i\in [d]:\; \forall_{l\in [K]}f_i(y^l(i))\ls f_i(j),\    \exists_{l\in [K]} \; f_i(y^l(i))=f_i(j) \}, 
\end{align*}
where the equalities in the above formulas holds, since $f_1,\ldots,f_d$ are strictly increasing.
\begin{lema}\label{lem:pier}
Fix $x\in \ccF$ and $y\in [n]^{Kd}$ which is bad. Then there exists a  number $\va(x,y)\gs 0$ with the following property. Define sets
\begin{align*}
J_{x,y}&:=\{i\ls d: t^{x}_if_i(x(i))>\va(x,y)\},\\
S_{x,y}&:=\cbr{i\in [d]:\; \exists_{l\ls K} \;\; y^l(i)\gs x(i)}=\cbr{i\in [d]:\; \exists_{l\ls K} \;\; f_i(y^l(i))\gs f_i(x(i))} =\cbr{i\in [d]:\; i\in S_{y}(x(i))}.
\end{align*}
Then $|S_{x,y}\cap J_{x,y}|\ls 1/2 |J_{x,y}|$.
\end{lema}
\begin{proof}
Consider a function $F:[0,\infty)\ra \R$,  given by
\[
F(\va)= \sum^d_{i=1}  \1_{i\in S_{x,y}} \min\rbr{t^x(i) f_i(x(i)),\va} - 1/2  \sum^d_{i=1} \min\rbr{t^x(i)  f_i(x(i)), \va}.
\]
By the definition of the set $S_{x,y}$, the family $\ccF$, and since $y\in [n]^{Kd}$ is bad, we have
\begin{equation}\label{eq:Witold}
  \sum^d_{i=1}  \1_{i\in S_{x,y}} t^x(i) f_i(x(i))\ls   \sum^d_{i=1}\sum^K_{l=1} t^{x}_i f_i(y^l(i)) < 1/2  \sum^d_{i=1} t^{x}_i  f_i(x(i)), 
\end{equation}
so we may define
\begin{equation}\label{eq:defeps}
\va(x,y)=\sup\{\va: F(\va)\gs 0 \}\in [0,\infty).    
\end{equation}
Now observe that for sufficiently small $\eta>0$ 
\begin{align*}
\sum^d_{i=1}  \1_{i\in S_{x,y}} \min\rbr{t^{x}_i f_i(x(i)),\va(x,y)+\eta}&= \sum^d_{i=1}  \1_{i\in S_{x,y}} \min\rbr{t^x(i) f_i(x(i)),\va(x,y)}+\eta \sum^d_{i=1}  \1_{i\in S_{x,y}\cap J_{x,y}}\\
\sum^d_{i=1} \min\rbr{t^x(i)  f_i(x(i)), \va(x,y)+\eta}&=\sum^d_{i=1} \min\rbr{t^x(i)  f_i(x(i)), \va(x,y)}+\eta \sum^d_{i=1} \1_{i \in J_{x,y}}.
\end{align*}
Since $F(\va(x,y)+\eta)<0$, the two identities above imply that
\[\sum^d_{i=1}  \1_{i\in S_{x,y}} \min\rbr{t^x(i) f_i(x(i)),\va(x,y)}+\eta \sum^d_{i=1}  \1_{i\in S_{x,y}\cap J_{x,y}} < 1/2  \sum^d_{i=1} \min\rbr{t^x(i)  f_i(x(i)), \va(x,y)}+1/2 \eta \sum^d_{i=1} \1_{i\in J_{x,y}}. \]
But $F(\va(x,y))\gs 0$ so the above inequality gives
\[ |S_{x,y}\cap J_{x,y}|=\sum^d_{i=1}  \1_{i\in S_{x,y}\cap J_{x,y}} < 1/2   \sum^d_{i=1} \1_{i\in J_{x,y}}=1/2|J_{x,y}|. \]
\end{proof}
\noindent For abbreviation, we define
\begin{align*}
  J_{x}(j)&:=\{i\in [d]:\; x(i)\gs j\}\\
  J_{x,y}(j)&:=J_{x,y}\cap J_{x}(j)=\{i\in [d]: t^x(i)f_i(x(i))>\va(x,y),\ x(i)\gs j\}\\
  J^=_{x,y}(j)&:=J_{x,y}(j)\setminus J_{x,y}(j+1)=  \{i\in [d]: t^x(i)f_i(x(i))>\va(x,y),\ x(i)=j\} \\
  S_{x,y}(j)&:=S_{x,y}\cap J_{x}(j)=\{i\in [d]: \exists_{l\ls K} y^l(i)\gs x(i),\ x(i)\gs j \}.
\end{align*}
Since for every $i\leq d$, the function $f_i$ is strictly increasing, we have
\begin{align*}
  J_{x}(j)&=\{i\in [d]:\; f_i(x(i))\gs f_i(j)\}\\
  J_{x,y}(j)&=\{i\in [d]: t^x(i)f_i(x(i))>\va(x,y),\ f_i(x(i))\gs f_i(j)\} \\
  J^=_{x,y}(j)&=J_{x,y}(j)\setminus J_{x,y}(j+1)=  \{i\in [d]: t^x(i)f_i(x(i))>\va(x,y),\ f_i(x(i))=f_i(j)\} \\
  S_{x,y}(j)&=S_{x,y}\cap J_{x}(j)=\{i\in [d]: \exists_{l\ls K} f_i(y^l(i))\gs f_i(x(i)),\ f_i(x(i))\gs f_i(j) \}.
\end{align*}

\noindent  Define also
\begin{align*}
  W_{x,y}:=J_{x,y}\cap S^c_{x,y}&=\{i\in [d]:t^{x}_if_i(x(i))>\va(x,y),\ \forall_{l\ls K} y^l(i)<x(i) \}\\
  &=\{i\in [d]:t^{x}_if_i(x(i))>\va(x,y),\ \forall_{l\ls K} f_i(y^l(i))<f_i(x(i)) \},\\
 W_{x,y}(j):=J_{x,y}(j) \cap  S^c_{x,y}&=\{i\in [d]:t^{x}_if_i(x(i))>\va(x,y),\ x(i)\gs j ,\ \forall_{l\ls K} y^l(i)<x(i)\}\\
 &=\{i\in [d]:t^{x}_if_i(x(i))>\va(x,y),\ f_i(x(i))\gs f_i(j) ,\ \forall_{l\ls K} f_i(y^l(i))<f_i(x(i))\}.
 \end{align*}
For any $x\in \ccF$,  $t^x\in T\subset \R^d_0$, so by the definition of $W_{x,y}$ we have
\begin{equation}\label{eq:naWsieniezeruje}
    \forall_{i\in W_{x,y}} f_i(x(i))>0 \ ( \textrm{equivalently }  \forall_{i\in W_{x,y}}  x(i)>1  ) . 
\end{equation}
The above equivalence holds since for each $i\leq d$ the function $f_i$ is strictly increasing and $f_i(1)=0$.

\noindent We leave it to the reader to verify that
 \begin{equation}\label{eq:gradacjaW}
   W_{x,y}(j)   = \bigcup^n_{l=j} J^{=}_{x,y}(l)\cap S^c_y(l).
 \end{equation}
Clearly $J_{x,y}(j)\supset J_{x,y}(j+1)$ and  $W_{x,y}(j)\supset W_{x,y}(j+1)$.
 Finally, we define
 \begin{align*}
 W^=_{x,y}(j)&:=W_{x,y}(j)\backslash W_{x,y}(j+1)=J^{=}_x(j) \cap S^c_y(j)=W_{x,y}(j) \cap J^{=}_{x,y}(j)\\
    &= \{i\in [d]:t^{x}_if_i(x(i))>\va(x,y),\ x(i)= j ,\ \forall_{l\ls K} \ y^l(i)<x(i)\}\\
    &=\{i\in [d]:t^{x}_if_i(x(i))>\va(x,y),\ f_i(x(i))= f_i(j) ,\ \forall_{l\ls K} \ f_i(y^l(i))<f_i(x(i))\}.
 \end{align*}
Obviously $W_{x,y}=\bigcup_{k=1}^n W^=_{x,y}(k)$ and the sets $(W^=_{x,y}(k))_{k=1,\ldots,n}$ are disjoint so
\begin{equation}\label{eq:locpikachu}
   \P(\forall_{i\in W_{x,y}} X_i=x(i))= \prod_{k=1}^n p_k^{|W^=_{x,y}(k)|}=\P(\forall_{i\in W_{x,y}} f_i(X_i)=f_i(x(i))), 
\end{equation}
where
\[p_k:=\P(X_i=k)=\P(f_i(X_i)=f_i(k)).\]
Also by the assumptions of the Lemma \ref{lem:prawdzly} $f_i(1)=0$ and  $\P(X_1=1)\gs 1-1/K$, so
\begin{equation}\label{eq:uciecie}
 p_1=\P(f_i(X_i)=0)=\P(f_i(Y^l(i))=0)\gs 1-1/K.
\end{equation}
\noindent
Since $f_i(1)=0$, we have that  
\begin{equation}\label{eq:w1zeruje}
    W^=_{x,y}(1)=\emptyset,
\end{equation}
and as a result $  \P(\forall_{i\in W_{x,y}} X_i=x(i))= \prod_{k=2}^n p_k^{|W^=_{x,y}(k)|}$ but we deliberately state \eqref{eq:locpikachu} as it is.\\

\begin{lema}\label{eq:wyliczonyZ}
Using the above notation
\[\rbr{W_{x,y}(j)\cup S_y(j) }\backslash \rbr{W_{x,y}(j+1)\cup S_y(j+1) }=W^=_{x,y}(j)\cup \rbr{S^{=}_y(j)\backslash W_{x,y}}.\]
\end{lema}

\begin{proof}
Clearly $S_y(j+1)\subset S_y(j)  $ and $W^=_{x,y}(j)\subset S^c_y(j)$ so $W^=_{x,y}(j) \cap S_y(j+1)=\emptyset$. This implies
\begin{equation}\label{eq:locc2}
 \rbr{W_{x,y}(j)\cup S_y(j) }\backslash \rbr{W_{x,y}(j+1)\cup S_y(j+1) }=W^=_{x,y}(j)\cup \rbr{ S^{=}_y(j)\backslash W_{x,y}(j+1)}.   
\end{equation}
By standard set-theoretic operations
\begin{align}
 S^{=}_y(j)\backslash W_{x,y}(j+1)&=S^{=}_y(j)\backslash \rbr{ J_{x,y}(j+1) \cap  S^c_{x,y}}\nonumber\\
  =S^{=}_y(j)\backslash \rbr{ J_{x,y} \cap J_x(j+1)  \cap S^c_{x,y} }&=\rbr{S^{=}_{y}(j)\cap J^c_{x,y}}\cup \rbr{S^{=}_y(j)\cap J^c_x(j+1)}
\cup \rbr{S^{=}_y(j)\cap S_{x,y}} \nonumber\\
=\rbr{S^{=}_{y}(j)\cap J^c_{x,y}}\cup \rbr{S^{=}_y(j)\cap S_{x,y}}&=\rbr{S^{=}_{y}(j)\setminus ( J_{x,y}\cap S^c_{x,y} )}=S^{=}_y(j)\backslash W_{x,y},\label{eq:loc3}
\end{align}
where  in the last line we have used that
\begin{align*}
S^{=}_y(j)\cap J^c_x(j+1)&=\left\{i\in [d]: \forall_{l\ls K} y^l(i)\ls j,\ \exists_{l\ls K}  y^l(i)=j,\ x(i)<j+1 \right\}\\
&=\left\{i\in [d]: \forall_{l\ls K} f_i(y^l(i))\ls f_i(j),\ \exists_{l\ls K}  f_i(y^l(i))=f_i(j),\ f_i(x(i))<f_i(j+1) \right\}\subset S_{x,y}.    
\end{align*}

The assertion follows by \eqref{eq:locc2} and \eqref{eq:loc3}.
\end{proof}

\noindent We introduce the pivotal definition for this paper.

\begin{defi}[Witness]\label{def:swiad}
 We say that $x'\in \ccF$ is admissible for $(x,y)$,  if 
\begin{equation}\label{eq:loccc1}
 \forall_{j\in [n]}  \forall_{i\in  W_{x',y}(j)} \ j\ls  x(i) \textrm{ (equivalently } f_i(j)\ls  f_i(x(i)))  .   
\end{equation}
Among all $x'\in \ccF$ that are $(x,y)$ admissible we choose the one for which the cardinality of $J_{x',y}$ is the smallest. Among the latter, we choose any $x'$ such that $W_{x',y}$ has a minimum number of elements. We denote the chosen element (with a little abuse of notation) by $x_\ast$. We refer to $x_\ast$ as a witness.
\end{defi}
\begin{rema}\label{rem:warrow}
In fact, \eqref{eq:loccc1} is equivalent to 
\[\forall_{i\in W_{x',y}} \ x'(i)\ls x(i)  \textrm{ (equivalently } f_i(x'(i))\ls f_i(x(i))) . \]
To see this, it is sufficient to observe that $W_{x',y}=W_{x',y}(1)$ (since $f_i(j)\gs f_i(1)$). However, we deliberately state \eqref{eq:loccc1} as it is, because it is more convenient. 
\end{rema}

\begin{fak}\label{fak:rodz}
Fix the bad $y\in [n]^{dK}$ and consider $\ccG(y):=\{(x_\ast,W_{x_\ast,y}): x\in \ccF\}$, where $x_\ast$ is the $(x,y)$ witness from Definition \ref{def:swiad}. Then $\ccG(y)$ is a cover of $\ccF$. In other words, the set of all witnesses defines a cover of $\ccF$. 
\end{fak}
\begin{proof}
Take any $x\in \ccF$. Witness $x_\ast$ is $(x,y)$ admissible (formula \eqref{eq:loccc1}) so that by \eqref{eq:naWsieniezeruje}  
\begin{equation}\label{eq:pokrycienieujemne}
 \forall_{i\in W_{x_\ast,y}}  0< f_i(x_\ast(i))\ls f_i(x(i)).
\end{equation}
\end{proof}

\noindent
We now introduce a technical lemma which is used in the proof of Lemma \ref{lem:dlugirach}.
\begin{lema}\label{lem:tech}
   Fix $x,x'\in \ccF$ and bad $y,y'\in [n]^{Kd}$. Let $x_\ast$ be a $(x,y)$ witness and $x'_\ast$ be a $(x',y')$ witness.  Assume that $|W_{x_\ast,y}|=|W_{x'_\ast,y'}|$. Then
    \[\sum^n_{j=1} \sum_{i\in J^{=}_{x_{\ast}',y'}(j)} \1_{i\in W_{x_{\ast}',y'}(j)\backslash S_{y'}}\gs\sum^n_{j=1} \sum_{i\in J^{=}_{x'_{\ast}}(j)} \1_{i\in W_{x_{\ast},y}(j)\backslash S_y(j)} \]
\end{lema}
\begin{proof}
    First observe that $W_{x_{\ast}',y'}(j)\cap S^c_{y'}(j)\cap J^{=}_{x_{\ast}',y'}(j)=W^=_{x'_\ast,y'}(j)$ (recall \eqref{eq:gradacjaW}) so by assumptions 
\begin{align*}
 \sum^n_{j=1} \sum_{i\in J^{=}_{x_{\ast}',y'}(j)} \1_{i\in W_{x_{\ast}',y'}(j)\backslash S_{y'}(j)}=\sum_{j=1}^n |W^=_{x'_\ast,y'}(j)|=|W_{x'_\ast,y'}|=|W_{x_\ast,y}|=\sum_{i=1}^d \sum_{j=1}^n \1_{i\in W^=_{x_\ast,y}(j)} .
\end{align*}
Now notice that
\begin{align*}
\sum^n_{j=1} \sum_{i\in J^{=}_{x'_{\ast}}(j)} \1_{i\in W_{x_{\ast},y}(j)\backslash S_y(j)}&\ls\sum^n_{j=1} \sum_{i\in J^{=}_{x'_{\ast}}(j)} \1_{i\in W_{x_{\ast},y}(j)}=\sum^n_{j=1} \sum_{i=1}^d \1_{i\in J^{=}_{x'_{\ast}}(j)}\sum_{l=j}^n \1_{i\in W^=_{x_\ast,y}(l)}\nonumber\\
&=\sum_{l=1}^n \sum_{i=1}^d \1_{i\in W^=_{x_\ast,y}(l)} \sum_{j=1}^l \1_{i\in J^{=}_{x'_{\ast}}(j)}\ls\sum_{l=1}^n \sum_{i=1}^d \1_{i\in W^=_{x_\ast,y}(l)}, 
\end{align*}
where the last inequality is true because the sets $(J^{=}_{x'_{\ast}}(j))_{j=1}^n$ are disjoint. 
\end{proof}

\begin{lema}\label{lem:dlugirach}
   Fix $x,x'\in \ccF$ and bad $y,y'\in [n]^{Kd}$. Let $x_\ast$ be a $(x,y)$ witness and $x'_\ast$ be a $(x',y')$ witness. Assume that
\begin{enumerate}
    \item $|J_{x_\ast,y}|=|J_{x'_\ast,y'}|$,
    \item $|W_{x_\ast,y}|=|W_{x'_\ast,y'}|$
    \item for any $j\ls n$ we have $S_y(j)\cup W_{x_\ast,y}(j)=S_{y'}(j)\cup W_{x'_\ast,y'}(j)$.
\end{enumerate}
Then $\va(x'_{\ast},y')\ls \va(x_{\ast}',y)$, where $\va(\cdot,\cdot)$ is the number from Lemma \ref{lem:pier}.
\end{lema}
\begin{proof}
By the definition of $\va(\cdot,\cdot)$  (see \eqref{eq:defeps}) it is enough to show that
\begin{equation}\label{goc1}
  \sum^d_{i=1} \1_{i\in S_y(x_{\ast}'(i))} \max\rbr{t^{x'_{\ast}}_i f_i(x_{\ast}'),\va(x'_{\ast},y')} \gs c \sum^d_{i=1} \max\rbr{t^{x_{\ast}'}_i  f_i(x_{\ast}'(i)), \va(x'_{\ast},y')}.   
\end{equation}

Using the definitions of the sets $S_{x'_\ast,y},J_{x'_\ast,y'}$ we get
\begin{align}\nonumber
& \sum^d_{i=1} \1 _{i\in S_{y}(x'_{\ast}(i))} \min\rbr{t^{x'_{\ast}}_i f_i(x'_{\ast}(i))),\va(x'_{\ast},y')}  \\
&=\va(x'_{\ast},y')  \sum^d_{i=1} \1_{i\in S_y(x_{\ast}'(i))}  
\1_{t^{x'_{\ast}}_i f_i(x_{\ast}'(i))>\va(x'_{\ast},y')}  +  \sum^d_{i=1} \1_{i\in S_y(x_{\ast}'(i))} t^{x'_{\ast}}_i f_i(x_{\ast}'(i)) \1_{t^{x'_{\ast}}_i f_i(x_{\ast}'(i))\ls \va(x'_{\ast},y')}\nonumber\\
&=\va(x'_{\ast},y') |J_{x_{\ast}',y'}\cap S_{x_{\ast}',y}| +\sum^d_{i=1} \1_{i\not\in J_{x_{\ast}',y'}}\1_{i\in S_{x_{\ast}',y}} t^{x'_{\ast}}_i f_i(x_{\ast}'(i)) \nonumber \\
& =\va(x'_{\ast},y')\cbr{|J_{x_{\ast}',y'}\cap S_{x_{\ast}',y'}| +|J_{x_{\ast}',y'}\cap S_{x_{\ast}',y}|-|J_{x_{\ast}',y'}\cap S_{x_{\ast}',y'}| } \nonumber  \\
&+ \sum^d_{i=1} \1_{i\not\in J_{x_{\ast}',y'}}\1_{i\in S_{x_{\ast}',y'}} t^{x'_{\ast}}_i f_i(x_{\ast}'(i)) 
+ \sum^d_{i=1} \1_{i\not\in J_{x_{\ast}',y'}}\cbr{\1_{i\in S_{x_{\ast}',y}}-\1_{i\in S_{x_{\ast}',y'}}}  t^{x'_{\ast}}_i f_i(x_{\ast}'(i)), \label{eq:loc4}
\end{align}
where in the last line we just added and subtracted appropriate  quantities. By the same argument (using the definitions of the sets $S_{x'_\ast,y'},J_{x'_\ast,y'}$)
\begin{multline*}
\va(x'_{\ast},y')|J_{x_{\ast}',y'}\cap S_{x_{\ast}',y'}|+ \sum^d_{i=1} \1_{i\not\in J_{x_{\ast}',y'}}\1_{i\in S_{x_{\ast}',y'}} t^{x'_{\ast}}_i f_i(x_{\ast}'(i))\\
= \sum^d_{i=1} \1_{i\in S_y'(x_{\ast}'(i))} \max\rbr{t^{x'_{\ast}}_i f_i(x_{\ast}'),\va(x'_{\ast},y')} \gs c  \sum^d_{i=1} \max\rbr{t^{x_{\ast}'}_i  f_i(x_{\ast}'(i)), \va(x'_{\ast},y')},  
\end{multline*}
where the last inequality follows from the definition of $\va(x'_\ast,y')$. Considering the above inequality, \eqref{eq:loc4} implies \eqref{goc1} (so the assertion), provided that
\begin{align}\label{eq:loc5}
 \va(x'_{\ast},y')\left(|J_{x_{\ast}',y'}\cap S_{x_{\ast}',y}|-|J_{x_{\ast}',y'}\cap S_{x_{\ast}',y'}| \right)+  \sum^d_{i=1} \1_{i\not\in J_{x_{\ast}',y'}}\cbr{\1_{i\in S_{x_{\ast}',y}}-\1_{i\in S_{x_{\ast}',y'}}}  t^{x'_{\ast}}_i f_i(x_{\ast}'(i)) \gs 0.
\end{align}
To show \eqref{eq:loc5} observe that
\begin{align*}
  J^c_{x'_\ast,y'}&=\bigcup_{j=1}^n \left\{i: f_i(x'_\ast(i))=f_i(j),\ t^{x'_\ast}f_i(x'_\ast(i))\ls\va(x'_\ast,y')\right\}=\bigcup_{j=1}^n J^=_{x'_\ast}(j)\setminus J^=_{x'_\ast,y'}(j),\\
  J_{x'_\ast,y'}&=\bigcup_{j=1}^n J^=_{x'_\ast,y'}(j)
\end{align*}
 (we recall that the sets $J^=_{x'_\ast,y'}(j)$ are disjoint as well as $J^=_{x'_\ast}(j)$). Also for any $i\in J^{=}_{x'_{\ast}}(j)$ (trivially $J^{=}_{x_{\ast}',y'}(j)\subset J^{=}_{x'_{\ast}}(j)$) we have $\1_{S_{x'_\ast,y}}-\1_{S_{x'_\ast,y'}}=\1_{S_{x'_\ast,y}(j)}-\1_{S_{x'_\ast,y'}(j)}$.
Using these observations we obtain
\begin{align}
 &\va(x'_{\ast},y')\left(|J_{x_{\ast}',y'}\cap S_{x_{\ast}',y}|-|J_{x_{\ast}',y'}\cap S_{x_{\ast}',y'}| \right)+  \sum^d_{i=1} \1_{i\not\in J_{x_{\ast}',y'}}\cbr{\1_{i\in S_{x_{\ast}',y}}-\1_{i\in S_{x_{\ast}',y'}}}  t^{x'_{\ast}}_i f_i(x_{\ast}'(i)) \nonumber \\
& =\va(x'_{\ast},y')\sum^n_{j=1}  \sum_{i\in J^{=}_{x_{\ast}',y'}(j)} \rbr{\1_{i\in S_y(j)}-  
\1_{i\in S_{y'}(j)} } 
+\sum^n_{j=1} \sum_{i\in J^{=}_{x'_{\ast}}(j)\backslash J^{=}_{x'_{\ast},y'}(j)}  
\cbr{\1_{i\in S_{y}(j)}-\1_{i\in S_{y'}(j)}   }t^{x'_{\ast}}_i f_i(x_{\ast}'(i))\nonumber \\
& =\va(x'_{\ast},y')\sum^n_{j=1} \sum_{i\in J^{=}_{x_{\ast}',y'}(j)} \left(\1_{i\in W_{x_{\ast}',y'}(j)\backslash S_{y'}(j)}-\1_{i\in W_{x_{\ast},y}(j)\backslash S_y(j)}\right)  \nonumber\\
& +\sum^n_{j=1} \sum_{i\in J^{=}_{x'_{\ast}}(j)\backslash J^{=}_{x'_{\ast},y'}(j)}  
\left(\1_{i\in W_{x_{\ast}',y'}(j)\backslash S_{y'}}-\1_{i\in W_{x_{\ast},y}(j)\backslash S_y(j)}\right)t^{x'_{\ast}}_i f_i(x_{\ast}'(i))\label{eq:kondlug}
\end{align}
where the second equality follows since by the assumption $3)$ of the lemma we have
\begin{align*}
 \1_{i\in S_y(j)}- \1_{i\in S_{y'}(j)}&= \1_{i\in  W_{x_{\ast},y}(j)\cup S_y(j)}-\1_{i\in W_{x_{\ast},y}(j)\backslash S_y(j)} -\left(\1_{i\in W_{x_{\ast}',y'}(j)\cup S_{y'}(j)}-\1_{i\in W_{x_{\ast}',y'}(j)\backslash S_{y'}(j)}\right)\\
 &=\1_{i\in W_{x_{\ast}',y'}(j)\backslash S_{y'}(j)}-\1_{i\in W_{x_{\ast},y}(j)\backslash S_y(j)}.
\end{align*}
If $i\in J^{=}_{x'_{\ast}}(j)\backslash J^{=}_{x'_{\ast},y'}(j)$  then $t^{x'_{\ast}}_i f_i(x_{\ast}'(i)) \ls \va(x'_\ast,y')$ so
\begin{align*}
&\va(x'_{\ast},y')\sum^n_{j=1} \sum_{i\in J^{=}_{x_{\ast}',y'}(j)} \left(\1_{i\in W_{x_{\ast}',y'}(j)\backslash S_{y'}(j)}-\1_{i\in W_{x_{\ast},y}(j)\backslash S_y(j)}\right)\\
&+\sum^n_{j=1} \sum_{i\in J^{=}_{x'_{\ast}}(j)\backslash J^{=}_{x'_{\ast},y'}(j)}  
\left(\1_{i\in W_{x_{\ast}',y'}(j)\backslash S_{y'}}-\1_{i\in W_{x_{\ast},y}(j)\backslash S_y(j)}\right)t^{x'_{\ast}}_i f_i(x_{\ast}'(i))  \\
&\gs \va(x'_\ast,y)\left(\sum^n_{j=1} \sum_{i\in J^{=}_{x_{\ast}',y'}(j)} \left(\1_{i\in W_{x_{\ast}',y'}(j)\backslash S_{y'}(j)}-\1_{i\in W_{x_{\ast},y}(j)\backslash S_y(j)}\right)\right.\nonumber\\
&\left.-\sum^n_{j=1} \sum_{i\in J^{=}_{x'_{\ast}}(j)\backslash J^{=}_{x'_{\ast},y'}(j)}  
\1_{i\in W_{x_{\ast},y}(j)\backslash S_y(j)}\right)  \\
&=\va(x'_\ast,y)\left(\sum^n_{j=1} \sum_{i\in J^{=}_{x_{\ast}',y'}(j)} \1_{i\in W_{x_{\ast}',y'}(j)\backslash S_{y'}}-\sum^n_{j=1} \sum_{i\in J^{=}_{x'_{\ast}}(j)} \1_{i\in W_{x_{\ast},y}(j)\backslash S_y(j)}\right)\gs 0, 
\end{align*}
where in the last line we used Lemma \ref{lem:tech}. Equality \eqref{eq:kondlug} and the inequality aboe imply \eqref{eq:loc5}.
\end{proof}

\begin{lema}\label{lem:zaw}
    Under the assumptions of Lemma \ref{lem:dlugirach} we have $W_{x_\ast,y}\subset J_{x'_\ast,y'}$.

\end{lema}
\begin{proof}
Take any $i\in W_{x'_\ast,y}(j)=\bigcup_{l=j}^n J^=_{x'_\ast,y}(l)\cap S^c_y(l)$ (recall formula \eqref{eq:gradacjaW}).  So $i\in J^=_{x'_\ast,y}(l)\cap S^c_y(l)$  for some fixed $l\gs j$.  
Thus 
\begin{equation}\label{for:iwlas}
    i\notin S_y(l), \textrm{ and } \va(x'_\ast,y)<t^{x'_\ast}_i f_i(x'_\ast(i)), \ x'_\ast(i)=l.
\end{equation}
Lemma \ref{lem:dlugirach} implies $\va(x'_\ast,y')\ls\va(x'_\ast,y)$ so  $i\in J^=_{x'_\ast,y'}(l)$. By dichotomy, \eqref{eq:gradacjaW} and the assumption $3)$ of Lemma \ref{lem:dlugirach}
\[i \in \left( J^=_{x'_\ast,y'}(l)\cap S^c_{y'}(l) \right) \cup S_{y'}(l)\subset W_{x'_\ast,y'}(l)\cup S_{y'}(l)    =W_{x_\ast,y}(l)\cup S_y(l).  \]
So we have that $i\in W_{x_\ast,y}(l)\subset W_{x_\ast,y}(j)$ (we recall \eqref{for:iwlas} and that by \eqref{eq:gradacjaW} the family $W_{x_\ast,y}(\cdot)$ is decreasing). As a result,
\begin{equation}\label{for:wzaw}
W_{x'_\ast,y}(j)\subset W_{x_\ast,y}(j).    
\end{equation}
This means that $x'_\ast$ is $(x,y)$ admissible (since $x_*$ is $(x,y)$ admissible by Definition \ref{def:swiad}).  By assumption $1)$ of Lemma \ref{lem:dlugirach} and by the optimal choice  of $x_\ast$ we get that
\[|J_{x'_\ast,y'}|=|J_{x_\ast,y}|\ls |J_{x'_\ast,y}|.\]
However, by Lemma \ref{lem:dlugirach}  $\va(x'_\ast,y')\ls\va(x'_\ast,y)$  so $J_{x'_\ast,y}\subset J_{x'_\ast,y'}$. As a result
\begin{equation}\label{eq:loc7}
 J_{x'_\ast,y}=J_{x'_\ast,y'}.    
\end{equation}
Since $x_\ast$ is a $(x,y)$ witness, we have that $|W_{x_\ast,y}|\ls|W_{x'_\ast,y}|$ (recall Definition \ref{def:swiad}). On the other hand  \eqref{for:wzaw} implies
\[W_{x'_\ast,y}=\bigcup_{j=1}^n W_{x'_\ast,y}(j)\subset\bigcup_{j=1}^n W_{x_\ast,y}(j)=W_{x_\ast,y}. \]
So in fact $W_{x_\ast,y}=W_{x'_\ast,y}$. The assertion follows because of \eqref{eq:loc7} $W_{x'_\ast,y}\subset J_{x'_\ast,y}=J_{x'_\ast,y'}$.    
\end{proof}

\begin{defi}\label{def:klas}
Let $j\gs t$ be natural numbers and $Z(1)\supset \ldots \supset Z(n)$ be a decreasing sequence of subsets of $[n]$.
By $\klas$ we denote all subsets of $[n]$ with the following property: $W\in \klas$ if and only if there exists $x\in \ccF$, $y\in [n]^{Kd}$ which is bad (recall Definition \ref{def:bady}) such that
\begin{enumerate}
    \item $W=W_{x_\ast,y}$,
    \item $j=|J_{x_\ast,y}|$,
    \item $t=|W_{x_\ast,y}|$, 
    \item for any $k=1,\ldots,n$ we have $Z(k)=S_y(k)\cup W_{x_\ast,y}(k),$
\end{enumerate}
where $x_\ast$ is the $(x,y)$ witness (recall Definition \ref{def:swiad}).
\end{defi}
\noindent Let $W\in \klas$, so that $W=W_{x_\ast,y}$ for some $x_\ast$ and bad $y\in [n]^{Kd}$. From \eqref{eq:wlasnoscif} we have that $W=W_{x_\ast,y}=W_{x_\ast,y}(1)$, so  that by the property $4)$ of Definition \ref{def:klas}
\begin{equation}\label{eq:WsiedziwZ}
    W\in \klas \Rightarrow W=W_{x_\ast,y}(1)\subset Z(1).
\end{equation}

\begin{rema}\label{rem:zakrespar}
The class $\klas$ is empty if $t<j/2$. To see this take any $W$ from this class and let $W=W_{x_\ast,y}$ where $x_\ast$ is a $(x,y)$ witness. Then by Lemma \ref{lem:pier}
\[t=|W|=|W_{x_\ast,y}|=|J_{x_\ast,y}\cap S^c_{x_\ast,y}|=|J_{x_\ast,y}|-|J_{x_\ast,y}\cap S_{x_\ast,y}|\gs 1/2|J_{x_\ast,y}|=j/2.\]
\end{rema}


\begin{rema}\label{rema:formuly}
Fix $j,t,Z(1),\ldots,Z(n)$ as in Definition \ref{def:klas}. Take any $W\in \klas$. Let $x\in \ccF$, $y\in [n]^{Kd}$ be bad and $x_\ast$ be a $(x,y)$ witness such that $W=W_{x_\ast,y}$. Then the structure of $x_\ast,y$ is determined by $W$. By Definition \ref{def:klas}, $Z(k)=W_{x_\ast,y}(k)\cup S_y(k)$, so Lemma \ref{eq:wyliczonyZ} implies
\[W^=_{x_\ast,y}(k)\cup (S^=_y(k)\setminus W_{x_\ast,y})=Z(k)\setminus Z(k+1)=:Z^=(k).\]
Since $W^=_{x_\ast,y}(k)\subset W_{x_\ast,y}=W$ we have
\begin{equation}\label{eq:locsnorlax}
\rbr{S^{=}_{y}(k)\backslash W_{x_\ast,y}}=Z^=(k) \setminus W \textrm{ and } W^=_{x_\ast,y}(k)=Z^=(k) \cap W.    
\end{equation}

\end{rema}

\begin{coro}\label{cor:zlicz}
Fix any $j,t$  and  $Z(1)\subset \ldots \subset Z(n)$ as in Definition \ref{def:klas}. Then
\[|\klas| \ls\binom{j}{t}.\]
\end{coro}
\begin{proof}
    Fix $W\in \klas$ and consider any other $W'$ from this class. By definition, there exist $x,x'\in \ccF$,  $y,y'\in[n]^{Kd}$ which are bad, and $x_\ast,x'_\ast$ respectively $(x,y)$ and $(x',y')$ witnesses   such that
    \[W=W_{x_\ast,y},\ W'=W_{x'_\ast,y'}.\]
    But the pairs $(x_\ast,y),(x'_\ast,y')$ satisfy the assumption of Lemma \ref{lem:zaw} (by Definition of $\klas$) which states that
    \[W'=W_{x'_\ast,y'}\subset J_{x_\ast,y}.\]
The assertion follows, since by the definition of the class $\klas$, $|W'|=t,\ |J_{x_\ast,y}|=j$.
\end{proof}

\begin{proof}[Proof of Lemma \ref{lem:prawdzly}]
The family $\ccF$ is not $\delta$-small (recall Definition \ref{def:smallfam}) and by Fact \ref{fak:rodz} for any bad $y\in[n]^{Kd}$ the set $\ccG(y)=\{(x_\ast,W_{x_\ast,y}) \mid x\in \ccF\}$ (we recall that $x_\ast$ is the $(x,y)$ witness cf. Definition \ref{def:swiad}) is a cover of $\ccF$, so
\[\delta \ls \sum_{x_\ast\in \ccG(y)}  \P(\forall_{i\in W_{x_\ast,y}}X_i= x_\ast(i)).\]
The above and \eqref{eq:locpikachu} yield (by $\textrm{bad}([n]^{Kd})$ we denote the set of all $y\in [n]^{Kd}$ which are bad)
 \begin{align}
& \delta \sum_{y\in \textrm{ bad}([n]^{Kd}) }\P(Y=y)\ls \sum_{y\in \textrm{ bad}([n]^{Kd})}\P(Y=y)  \sum_{(x_\ast,W_{x_\ast,y})\in \ccG(y)}  \P(\forall_{i\in W_{x_\ast,y}}X_i= x_\ast(i))\nonumber \\
&= \sum_{y\in \textrm{ bad}([n]^{Kd})}\P(Y=y)  \sum_{(x_\ast,W_{x_\ast,y})\in \ccG(y)} \prod^n_{k=1} p_k^{|W^=_{x_{\ast},y}(k)|}\nonumber\\
& =\sum_{j=1}^n\sum_{t\gs j/2}^j \sum_{(Z(k))_{k=1}^n}  \sum_{W\in \klas}  \prod^n_{k=1} p_k^{|Z^=(k)\cap W|}\sum_{y\ra W}  \P(Y=y) ,\label{eq:loc1}
\end{align}
where in the last line we sum over all decreasing families $(Z(k))_{k=1}^n$ of subsets of $[n]$. We also used the notation
\[y\to W := \{y\in \textrm{bad}([n]^{Kd}):\ \exists_{x\in \ccF} W=W_{x_\ast,y},\ x_\ast \textrm{ is } (x,y) \textrm{ witness} \}.\]
Let us explain the last equality in \eqref{eq:loc1} which is a discrete change of variables. The sets $W$ from all the covers $(\ccG(y))_{y\in \textrm{bad}([n]^{Kd})}$ were grouped into disjoint classes $\klas$. Each such class consists of these sets $W$ which have the same specific parameters: $j=|J_{x_\ast,y}|,t=|W_{x_\ast,y}|,(Z(k))_{k=1}^n$ (we recall that $Z(k)=S_y(k)\cup W_{x_\ast,y}(k)$). By Remark \ref{rem:zakrespar}, $j/2 \ls t \ls j$ (otherwise $\klas$ is empty). Then we sum over all $y\subset [n]^{Kd}$ which are bad and led to the creation of the witness $W$. We  also used the formula \eqref{eq:locsnorlax} for $W^=_{x_\ast,y}(k)$ in terms of the new parameters. Also, by formula \eqref{eq:locsnorlax}
\begin{align}
\sum_{y\ra W}  \P(Y=y)&=\P(Y\ra W)\ls\P(\forall_{k\ls n} S^=_Y(k)\cap W^c = Z^=(k) \cap W^c)\nonumber \\
&=\prod_{k=1}^n\rbr{\rbr{\sum^k_{i=1} p_i}^{K}-\rbr{\sum^{k-1}_{i=1}p_i}^{K} }^{|Z^=(k)\setminus W|}.\label{eq:loc2} 
\end{align}
In the last equality we just use that $(Y^l(i))_{l\ls K, i\ls d}$ are i.i.d. with the same distribution as $X_1$. So the events $(S^=_Y(k)\cap W^c = Z^=(k) \cap W^c)_{k=1}^n$ are independent (since $Z^=(k) \cap W^c$ are disjoint). The rest is just a simple computation of the probability of a certain event. Now, by \eqref{eq:uciecie}
\[
\rbr{p_1+p_2+\ldots p_k}^K-\rbr{p_1+\ldots +p_{k-1}}^K\gs Kp_k \rbr{p_1+\ldots+p_{k-1}}^{K-1}\gs Kp_k (1-1/\cons)^{K-1}\gs Kp_k/e.
\]
This implies that 
\[ \prod^n_{k=1} p_k^{|Z^=(k)\cap W|}\ls\left( \frac{e}{K}\right)^t \prod^n_{k=1}  \rbr{\rbr{\sum^k_{i=1} p_i}^{K}-\rbr{\sum^{k-1}_{i=1}p_i}^{K} }^{|Z^=(k)\cap W|}\]
(from \eqref{eq:WsiedziwZ} we have that $t=|W|=\sum_{k=1}^n |Z^=(k)\cap W|$).
Plugging the above and \eqref{eq:loc2} into \eqref{eq:loc1} yields
\begin{align*}
&\delta \sum_{y\in \textrm{ bad}([n]^{Kd})}\P(Y=y)\\
&\ls    \sum_{j=1}^n\sum_{t\gs j/2}^j \sum_{(Z(k))_{k=1}^n}  \sum_{W\in \klas}  \left( \frac{e}{K}\right)^t \prod_{k=1}^n\rbr{\rbr{\sum^k_{i=1} p_i}^{K}-\rbr{\sum^{k-1}_{i=1}p_i}^{K} }^{|Z^=(k)|} \\
&\ls \sum_{j=1}^n\sum_{t\gs j/2}^j  \binom{j}{t}  \left( \frac{e}{K}\right)^t \sum_{(Z(k))_{k=1}^n}  \prod_{k=1}^n\rbr{\rbr{\sum^k_{i=1} p_i}^{K}-\rbr{\sum^{k-1}_{i=1}p_i}^{K} }^{|Z^=(k)|}, 
\end{align*}
where we used Corollary \ref{cor:zlicz} in the last line. Also, observe that
\[1\ge \sum_{(Z(k))_{k=1}^n} \P\left(\bigcap_{k=1}^n S^=_Y(k)=Z^=(k)\right)=\sum_{(Z(k))_{k=1}^n}\prod_{k=1}^n\rbr{\rbr{\sum^k_{i=1} p_i}^{K}-\rbr{\sum^{k-1}_{i=1}p_i}^{K} }^{|Z^=(k)|}. \]
So, in fact 
\begin{align*}
\delta \sum_{y\in \textrm{ bad}([n]^{Kd})}\P(Y=y) &\ls \sum_{t=1}^\infty \left( \frac{e}{K}\right)^t  \sum_{j=t}^{2t}  \binom{j}{t} = \sum_{t=1}^\infty \left( \frac{e}{K(1-1/C)}\right)^t  \binom{2t+1}{t+1}  \\
 &\ls \sum_{t=1}^\infty \left( \frac{4e}{K}\right)^t=\frac{4e}{K-4e}\ls \frac{\delta}{2},
\end{align*}
for (we recall that $\delta\ls 1$) $K\geq \lceil 12e / \delta\rceil$.
\end{proof}

\section*{Glossary}
\begin{center}
\begin{itemize}
    \item $J_x(j)=\{i\in [d]:f_i(x(i))\gs f_i(j)\}$,
    \item $J^=_x(j)=J_x(j)\setminus J_x(j+1)=\{i\in [d]:f_i(x(i))=f_i(j)\}$,
    \item $S_y(j)=\{i\in [d]:\exists \ l\ls K \; \; f_i(y^l(i))\gs f_i(j)\},$
    \item $ S^c_y(j)=\{i\in [d]:\forall_{l\ls K} \; \; f_i(y^l(i))< f_i(j)\}$,
    \item  $S^{=}_y(j):=S_{y}(j)\backslash S_y(j+1)=\{ i\in [d]:\; \forall_{l\in [K]}f_i(y^l(i))\ls f_i(j),\    \exists_{l\in [K]} \; f_i(y^l(i))=f_i(j) \}$,
    \item $S_{x,y}=\{i\in [d]: i\in S_y(x(i))\}=\{i\in [d]:\exists_{l\ls K} \; \; f_i(y^l(i))\gs f_i(x(i))\}$,
    \item $J_{x,y}=\{i\in [d]: t^x(i)f_i(x(i))>\va(x,y)$,
    \item $W_{x,y}=J_{x,y}\cap S^c_{x,y}=\{i:t^x(i)f_i(x(i))>\va(x,y),\forall_{l\ls K} \; \; f_i(y^l(i))\gs f_i(x(i))\}$,
    \item $J_{x,y}(j)=J_{x,y}\cap J_x(j)=\{i\in [d]: t^x(i)f_i(x(i))>\va(x,y),f_i(x(i))\gs f_i(j)\}$,
    \item $J^=_{x,y}(j)=J_{x,y}(j)\setminus J_{x,y}(j+1)=\{i\in [d]: t^x(i)f_i(x(i))>\va(x,y),f_i(x(i))=f_i(j)\}$,
    \item $S_{x,y}(j)=S_{x,y}\cap J_x(j)=\{i\in [d]: \exists_{l\ls K} \; \; f_i(y^l(i))\gs f_i(x(i)), f_i(x(i))\gs f_i(j)\}$,
    \item $W_{x,y}(j)=J_{x,y}(j)\cap S^c_{x,y}=\{i\in [d]: t^x(i)f_i(x(i))>\va(x,y), f_i(x(i))\gs f_i(j), \forall_{l\ls K} f_i(y^l(i))<f_i(x(i))\}$,
    \item $W^=_{x,y}(j)=W_{x,y}(j)\setminus W_{x,y}(j+1)=J^=_{x,y}(j)\cap S^c_y(j)=J^=_{x,y}\cap W_{x,y}$,
    \item $W^=_{x,y} = \{i\in [d]:t^{x}_if_i(x(i))>\va(x,y),\ f_i(x(i))= f_i(j) ,\ \forall_{l\ls K} f_i(y^l(i))<f_i(x(i))\}$.

  \end{itemize}
\end{center}

\end{document}